\documentclass{amsart}
\usepackage{amssymb, amsxtra, amsmath, amsthm, enumerate, enumitem, cite, bm}
\usepackage[a4paper,textwidth=352.2pt,textheight=547pt,centering]{geometry}
\AtBeginDocument{\setlength{\vsize}{\textheight}}

\author{Sanda Bujačić Babić}
\address{Faculty of Mathematics\\
University of Rijeka\\
Radmile Matejčić 2\\
51000 Rijeka\\
Croatia}
\email{sbujacic@math.uniri.hr}

\author{Ana Jurasić}
\address{Faculty of Mathematics\\
University of Rijeka\\
Radmile Matejčić 2\\
51000 Rijeka\\
Croatia}
\email{ajurasic@math.uniri.hr}



\thanks{This work is supported by the Croatian Science Foundation under the project number HRZZ IP-2022-10-5008 and by the European Union-NextGeneration EU under the project number uniri-iz-25-62-ALGEBRA}

\title{Regularity of Diophantine quadruples over $\mathbb{Q}(i)[X]$}

\theoremstyle{plain}
\newtheorem{thm}{Theorem}[section]
\newtheorem{lem}{Lemma}[section]

\newtheorem{cor}{Corollary}[section]

\theoremstyle{definition}
\newtheorem{definition}{Definition}[section]

\theoremstyle{remark}
\newtheorem{remark}{Remark}[section]
\newtheorem{ex}{Example}[section]

\begin{document}

\begin{abstract}
We prove that every Diophantine quadruple \(\{a,b,c,d\}\) over
\(\mathbb{Q}(i)[X]\) that contains at least one non-constant polynomial is regular;
that is, $$(a+b-c-d)^2=4(ab+1)(cd+1).$$  This result is consistent
with the corresponding results over \(\mathbb{Z}[i][X]\) and
\(\mathbb{R}[X]\), but contrasts with the situation over
\(\mathbb{C}[X]\).
\end{abstract}

\maketitle

\section{Introduction}
Diophantine $m$-tuples, $m \in \mathbb{N}$, are a classical topic in the field of number theory and go back to the work of Diophantus of Alexandria. He investigated sets of numbers in which each pair of different elements satisfies a certain algebraic condition. Over the centuries, the study of such sets has been generalized in many different ways. More formally, we call a set of $m$ distinct positive integers $\{a_1, a_2, \dots, a_m\}$ a Diophantine $m$-tuple with property $D(n)$, or simply a $D(n)$-$m$-tuple, for a non-zero integer $n$, if $a_i a_j +n$ is a perfect square for all $1\leq i < j \leq m$. 

The classical and best-known case introduced by Diophantus is the case in which $n = 1$. A $D(1)$-$m$-tuple is simply called the Diophantine $m$-tuple. Diophantus found such a set consisting of four different positive rational numbers, while the famous Diophantine quadruple $\{1, 3, 8, 120\}$ was found by Fermat and was the first such set of integers. Over the centuries, this study has undergone significant generalizations that include more comprehensive algebraic structures, varying constants and conditions for higher powers. To begin with, we present some basic definitions, remarks and results that we use in our work.

\begin{definition}
Let $m\geq 2$ and let $R$ be a commutative ring with unity. Let $n\in R$ be a non-zero element and $\{a_1,a_2,\dots,a_m\}$ a set of $m$ distinct non-zero elements in $R$ such that $a_ia_j+n$ is a square of an element in $R$ for $1\leq i<j\leq m$. The set $\{a_1,a_2,\dots,a_m\}$ is called a Diophantine $m$-tuple with the property $D(n)$ or simply a $D(n)$-$m$-tuple in $R$.
\end{definition}

\begin{remark}A $D(n)$-$m$-tuple in which the elements are not necessarily distinct and may include zero is referred to as an improper $D(n)$-$m$-tuple.
\end{remark}
\begin{remark}
 The case where $R$ is a polynomial ring and $n$ is a constant polynomial is usually called a polynomial $D(n)$-$m$-tuple. In such a variant of the problem it is usually assumed that not all polynomials in a $D(n)$-tuple are constant.    
\end{remark}

We study the polynomial $D(1)$-problem over the ring $\mathbb{Q}(i)[X]$. A polynomial generalization of this problem was first investigated by Jones for $n=1$ and the set of polynomials with integer coefficients \cite{jones1, jones2}. Afterwards, various polynomial generalizations of the original problem of Diophantus were also considered (for example \cite{dif2,dift,difw,dij,dl_17, jaialan, fij}).  

First, we introduce the definitions of regular Diophantine triples and Diophantine quadruples in $\mathbb{Q}(i)[X]$, which are analogous to those in other polynomial rings. Let $\{a,b,c\}$ be a Diophantine triple in $\mathbb{Q}(i)[X]$ such that 
\begin{equation}\label{jdbe_osnovne}
    ab+1=r^2,\ ac+1=s^2,\ bc+1=t^2,
\end{equation}
where $r,s,t\in \mathbb{Q}(i)[X]$.
In general, the problem of extending a Diophantine pair \( \{a, b\} \) to a Diophantine triple \( \{a, b, c\} \) is closely related to the solution of a certain Pellian equation. More precisely, for the pair \( \{a, b\} \), such that $ab+1=r^2$ holds for some element $r\in\mathbb{Q}(i)[X]$, from (\ref{jdbe_osnovne}) the equation
\[
bs^2 - at^2 = b - a
\]
can be formulated. Every extension of the pair $\{a,b\}$ to a Diophantine triple $\{a,b,c\}$ gives a polynomial solution $(s,t)$ of the Pellian equation. Conversely, a polynomial solution $(s,t)$ yields an extension whenever $c=\frac{s^2-1}{a}=\frac{t^2-1}{b} \in \mathbb{Q}(i)[X].$
Two distinguished candidates are obtained by setting $c=c_{\pm}=a+b\pm 2r,
\ r^2=ab+1.$ Whenever $c_{\pm}$ is nonzero and distinct from $a$ and $b$, it yields a regular extension of the Diophantine pair $\{a, b\}$ to a Diophantine triple $\{a,b,c_{\pm}\}$.

\begin{definition}\label{expc}
A Diophantine triple $\{a,b,c\}$ in $\mathbb{Q}(i)[X]$ is called regular if $(c-b-a)^2=4(ab+1).$
Equivalently,
\begin{align}
    c&=c_{\pm}=a+b\pm 2r,\label{jdba_c_regularni}\\
    ac_{\pm}+1 &=(a\pm r)^2,\ bc_{\pm}+1=(b\pm r)^2.\label{jdba_c_regularni1}
\end{align}
\end{definition}
\begin{ex}\label{exmp1.1}
The following is a regular Diophantine triple $\{a, b, c\}$ in $\mathbb{Q}(i)[X]$:
\[
\{a, b, c\}=\left\{
1,\ 
X^2 + (1+i)X + \frac{i}{2} - 1,\ 
X^2 + (3+i)X + \frac{3}{2}i+1
\right\}.
\]
\end{ex}

Besides regular extensions of a pair to a triple, there also exist Diophantine triples that are not regular. Similarly, every Diophantine triple $\{a,b,c\}$ gives rise to two explicit candidates $d_{+}$ and $d_{-}$ for a regular extension. Whenever one of these candidates is nonzero and distinct from $a, b$ and $c$, it yields a regular Diophantine quadruple.

\begin{definition}
A Diophantine quadruple $\{a,b,c,d\}$ in $\mathbb{Q}(i)[X]$ is called regular if $(a+b-c-d)^2=4(ab+1)(cd+1),$ 
or equivalently if 
\begin{equation}\label{jdba_d_pm}
    d=d_{\pm}=a+b+c+2(abc\pm rst), 
\end{equation}
where \(r,s,t\) are defined by \eqref{jdbe_osnovne}.
In this case, we have
\begin{align}\label{uvw}
&ad_{\pm}+1 = \left(rs \pm at\right)^2 = u_{\pm}^2, \ \ bd_{\pm}+1 =\left( rt \pm bs\right)^2=v_{\pm}^2,  \ \  cd_{\pm}+1 = \left(st \pm cr\right)^2=w_{\pm}^2,\nonumber
\end{align}
where $u_{\pm}, v_{\pm}, w_{\pm} \in \mathbb{Q}(i)[X]$.
\end{definition}

\begin{ex}
By applying formula \eqref{jdba_d_pm}, we construct the regular extension $\{a,b,c,d\}$ of the $D(1)$-triple $\{a, b, c\}$ introduced in Example \ref{exmp1.1}.  The resulting regular $D(1)$-quadruple in $\mathbb{Q}(i)[X]$ is given by
\[
\begin{aligned}
\bigg\{ 1, \ & X^2 + (1+i)X + \frac{i}{2} - 1, \ X^2 + (3+i)X + \frac{3}{2}i + 1, \\
& 4X^4 + (16+8i)X^3 + (12+24i)X^2 + (-8+16i)X - 5 \bigg\}.
\end{aligned}
\]
\end{ex}

\begin{remark}
Whenever the four polynomials
\begin{equation}\label{jdba_par_do_cetvorke}
\{a,b,a+b\pm 2r, 4r(r\pm a)(b\pm r) \} \end{equation}are nonzero and pairwise distinct, they form a regular Diophantine quadruple in $\mathbb{Q}(i)[X]$.
\end{remark}

\begin{definition}
An irregular Diophantine quadruple in $\mathbb{Q}(i)[X]$ is one that is not regular.
\end{definition}

Dujella and Luca proved in \cite[Lemma 1]{dl_17} that in an algebraically closed field \( K \) of characteristic zero at most one element of a polynomial Diophantine $m$-tuple can be a constant polynomial. An adapted important result for the Diophantine $m$-tuples in $\mathbb{C}[X]$ can be derived from \cite[Lemma 1]{dl_17}.

\begin{lem}\label{lema_samo_jedan_konstantan}
Let $\{a_1,a_2,\dots,a_m\}$ be a Diophantine $m$-tuple  in $\mathbb{C}[X]$, such that not all $a_i$'s are constant polynomials. Then $a_i\neq a_j$, for $i\neq j$, and at most one of the polynomials $a_i$, $i=1,\dots,m$, is constant.  
\end{lem} 

Hence, throughout the paper, we consider Diophantine tuples containing at least one non-constant polynomial; in such tuples, the elements are distinct and at most one of them is constant.

There are some useful consequences of Lemma~\ref{lema_samo_jedan_konstantan} for polynomials \(a\) and \(b\) of a Diophantine pair in \(\mathbb Q(i)[X]\). Suppose that \(ab=y^2\), where \(y\in\mathbb Q(i)[X]\). Together with \(ab+1=r^2\), this leads to
$(r-y)(r+y)=1.$ Hence both \(r-y\) and \(r+y\) are constant polynomials, and so \(r\) and \(y\) are constant. Therefore \(ab=y^2\) is constant, which implies that both \(a\) and \(b\) are constant, contrary to the assumption that at least one of them is non-constant. Also, if $b=c_1 a$, for some $c_1 \in \mathbb{Q}(i)\backslash \{0\}$, then from \( ab + 1 = r^2 \) we would have \begin{equation}
a^2 c_1 + 1 = r^2.\label{a_i_c_1}
\end{equation}  In $\mathbb{C}[X]$ we get a factorization of the form $1 = (r-a\sqrt{c_1})(r+a\sqrt{c_1}).$ Both factors must be constant polynomials. Hence, by subtracting these two constant factors, \(a\) is constant. Since \(b=c_1a\), \(b\) is constant as well, contrary to the assumption that at least one of \(a\) and \(b\) is non-constant. We therefore have the following lemma.
\begin{lem}
Let $\{a,b\}$ be a Diophantine pair in $\mathbb{Q}(i)[X]$, and assume that at least one of $a$ and $b$ is non-constant. Then
\begin{itemize}
    \item $ab$ cannot be a perfect square;
    \item there is no constant \( c_1 \in \mathbb{Q}(i)\backslash \{0\} \) such that \( b = c_1 a \).
\end{itemize} 
\end{lem}

As in the integer case, the problems with $n=1$ and $n=4$ are closely connected. For every $m\geq2$, a set $\{a_1,\ldots,a_m\}$ is a $D(4)$-$m$-tuple in $\mathbb{C}[X]$ (or in $\mathbb{Q}(i)[X]$) if and only if $\left\{\frac{a_1}{2},\ldots,\frac{a_m}{2}\right\}$
is a $D(1)$-$m$-tuple in the same ring. Moreover, this correspondence preserves regularity.
The results obtained in this paper for \(D(1)\)-quadruples in
\(\mathbb Q(i)[X]\) therefore also apply to \(D(4)\)-quadruples in
\(\mathbb Q(i)[X]\). Hence, our current work is motivated both by our earlier research on \(D(1)\)-quadruples and on \(D(4)\)-quadruples over various polynomial rings. 

From \cite{dij}, we know that there are at most $7$ elements in a Diophantine $m$-tuple in $\mathbb{Q}(i)[X]$, but we can also transform the result from \cite[Lemma 5]{dij}:

\begin{lem}\label{lema_supanj_delta}
Let $\{a,b,c,d\}$ be a Diophantine quadruple in $\mathbb{Q}(i)[X]$. Denote by $\alpha$, $\beta$, $\gamma$ and $\delta$ the degrees of $a$, $b$, $c$, $d$, respectively, and assume that $\alpha\leq\beta\leq\gamma\leq\delta$. Then
either $\delta\geq\frac{3\beta+5\gamma}{2}$ or $d = d_+$.\end{lem}

Filipin and Fujita \cite{ff} proved that every polynomial $D(4)$-quadruple in $\mathbb{Z}[X]$ is regular. Bliznac Trebje\v{s}anin and Buja\v{c}i\'{c} Babi\'{c} \cite{blizbuj} proved that every $D(4)$-quadruple in $\mathbb{Z}[i][X]$ is regular, where \(\mathbb{Z}[i]\) denotes the Gaussian integers. The same assertion is proved by Filipin and Jurasi\'{c} \cite{glavni} for $D(1)$-quadruples in $\mathbb{Z}[i][X]$. A similar goal is pursued in the present work, in which we investigate whether every Diophantine quadruple in $\mathbb{Q}(i)[X]$ containing at least one non-constant polynomial is regular and, consequently, whether the same holds for D(4)-quadruples in $\mathbb{Q}(i)[X]$. Since Filipin and Jurasi\'{c} \cite{fij} proved that every $D(1)$-quadruple in $\mathbb{R}[X]$  containing at least one non-constant polynomial is regular, it follows that every $D(4)$-quadruple in $\mathbb{R}[X]$ containing at least one non-constant polynomial is also regular. On the other hand, Dujella and Jurasi\'{c} \cite{dij} found irregular $D(1)$-quadruples in $\mathbb{C}[X]$, so there also exist irregular $D(4)$-quadruples in $\mathbb{C}[X]$. 
Thus, the case \(\mathbb Q(i)[X]\) complements the known results over \(\mathbb R[X]\), \(\mathbb C[X]\), and \(\mathbb Z[i][X]\).

In order to determine whether every Diophantine quadruple and, consequently, every $D(4)$-quadruple in $\mathbb{Q}(i)[X]$ is regular, certain techniques developed in \cite{blizbuj, dij, dl_17, glavni} serve as the basis for our analysis, but additional methods and refined strategies were required to capture the full range of cases occurring in this setting. As we will explain later, the structure of the field \(\mathbb Q(i)\) gives rise to additional cases that do not occur in \(\mathbb R[X]\) or \(\mathbb Z[i][X]\). As usual, we begin by analyzing the associated system of simultaneous Pellian equations and then determine the intersection of the resulting binary recursive sequences. We then apply congruence relations and the gap principle to arrive at the main result:

\begin{thm}\label{glavni-tm}Every Diophantine quadruple in $\mathbb{Q}(i)[X]$ containing at least one non-constant polynomial is regular.\end{thm}
\begin{cor}\label{glavni-cor}Every $D(4)$-quadruple in $\mathbb{Q}(i)[X]$ containing at least one non-constant polynomial is regular.\end{cor}

\section{General properties and Pellian equations} 
\label{general_pell}

Let \( \{a, b, c\} \) be a Diophantine triple for which equations (\ref{jdbe_osnovne}) hold. If such a triple is extended to a Diophantine quadruple \( \{a, b, c, d\} \) in $\mathbb{Q}(i)[X]$, with a non-zero polynomial $d$ such that $d\neq a, b, c,$ then there exist \( x, y, z \in \mathbb{Q}(i)[X] \) satisfying
\begin{equation}\label{jdb_jednakosti_za_d}
ad+1=x^2,\ bd+1=y^2,\ cd+1=z^2.
\end{equation}

Eliminating $d$ from (\ref{jdb_jednakosti_za_d}), we obtain the system of simultaneous Pellian equations
\begin{align}
    az^2-cx^2&=a-c,\label{jdba_pellova_prva}\\
    bz^2-cy^2&=b-c.\label{jdba_pellova_druga}
\end{align}
To compare the solutions of the system (\ref{jdba_pellova_prva})-(\ref{jdba_pellova_druga}) within the polynomial ring $\mathbb{Q}(i)[X]$, we use the ordering induced by the degrees of the polynomials. 
We denote by \( \alpha, \beta, \gamma, \delta \) the degrees of the polynomials \( a, b, c, d\), respectively. Without loss of generality, after relabelling if necessary, we may assume \( 0 \leq \alpha \leq \beta \leq \gamma \leq \delta \) and \( \beta, \gamma, \delta > 0 \). Lemma~\ref{lema_samo_jedan_konstantan} ensures that at most one of these
polynomials can be constant; if such a polynomial exists, we denote it by \(a\).

When dealing with equations of type (\ref{jdba_pellova_prva}) and (\ref{jdba_pellova_druga}), we first try to describe the set of all possible solutions $(z,x)$ and $(z,y)$, respectively. Among all the solutions, we are interested in those for which equations (\ref{jdb_jednakosti_za_d}) hold. Lemmas \ref{prvi_deg} and \ref{lema_rj_pellove_jedn} are adapted formulations of \cite[Lemma 4]{dl_17} and \cite[Lemma 2.1]{fij}, respectively. While originally stated in a different setting, their structure and conclusions remain applicable in the present context with appropriate adjustments.

\begin{lem}\label{prvi_deg}
Let $(z,x)$ and $(z,y)$ be solutions of (\ref{jdba_pellova_prva}) and (\ref{jdba_pellova_druga}), respectively, with $x,y,z\in \mathbb{Q}(i)[X]$. Then, there exist solutions $(z_0,x_0)$ and $(z_1,y_1)$, $z_0,z_1,x_0,y_1\in\mathbb{Q}(i)[X]$, of (\ref{jdba_pellova_prva}) and (\ref{jdba_pellova_druga}), respectively, such that
\begin{equation}\label{deg1}    \deg(z_0)\leq \frac{3\gamma-\alpha}{4},   \quad \deg(x_0)\leq \frac{\alpha+\gamma}{4},
\end{equation}
\begin{equation}\label{ineq20}
    \deg(z_1)\leq \frac{3\gamma-\beta}{4}, \quad \deg(y_1)\leq \frac{\beta+\gamma}{4}. 
\end{equation}
\end{lem}

\begin{lem}\label{lema_rj_pellove_jedn}
Let $(z,x)$ and $(z,y)$ be solutions of (\ref{jdba_pellova_prva}) and (\ref{jdba_pellova_druga}), respectively, with $x,y,z\in \mathbb{Q}(i)[X]$. Then, there exist non-negative integers $m$ and $n$ such that
\begin{equation}
        \label{jdba_opce_rjesenje_PRVE}
        z\sqrt{a}+x\sqrt{c}=(z_0\sqrt{a}+x_0\sqrt{c})\left(s+\sqrt{ac}\right)^{m},
\end{equation}
\begin{equation}
        \label{jdba_opce_rjesenje_DRUGE}
        z\sqrt{b}+y\sqrt{c}=(z_1\sqrt{b}+y_1\sqrt{c})\left(t+\sqrt{bc}\right)^{n}.
\end{equation}
\end{lem}

Let $(v_m)_{m\in \mathbb{Z}}$ and $(x_m)_{m\in \mathbb{Z}}$ be the sequences of polynomials given by (\ref{jdba_opce_rjesenje_PRVE}), such that
$v_0 = z_0$,
\begin{equation*}
v_1\sqrt{a}+x_1\sqrt{c} = (z_0\sqrt{a}+x_0\sqrt{c})\left(s+\sqrt{ac}\right), \ \text{and}
\end{equation*}
\begin{equation*}
v_{-1}\sqrt{a}+x_{-1}\sqrt{c} = (z_0\sqrt{a}+x_0\sqrt{c})\left(s+\sqrt{ac}\right)^{-1} = (z_0\sqrt{a}+x_0\sqrt{c})\left(s-\sqrt{ac}\right).
\end{equation*}
More precisely, we have
\begin{equation}
v_1 = z_0s+x_0c, \ \  v_{-1} = z_0s-x_0c.
\end{equation}
The solution \((z_0,x_0)\) is chosen, as in Lemma~\ref{prvi_deg}, with
minimal possible degree of \(z_0\). We conclude
\begin{align}
  \deg(v_1), \ \deg(v_{-1})\geq \deg(v_0) = \deg(z_0). 
\end{align}
Replacing \(x_0\) by \(-x_0\) interchanges \(v_1\) and \(v_{-1}\). Thus, by
allowing both choices $v_1=sz_0\pm cx_0,$ it is enough to consider indices \(m\geq0\).

Using induction on \(m\geq0\) in \eqref{jdba_opce_rjesenje_PRVE}, one proves that if \(c\mid(z^2-1)\), then \(c\mid(z_0^2-1)\) as well. The proof follows the strategy of \cite[Lemma 4]{dl_17}; the corresponding statement for \(z_1\) is obtained analogously. Thus, there exist $d_0, d_1\in\mathbb{Q}(i)[X]$ such that 
\begin{equation}\label{d0}
    ad_0+1 = x_0^2 \ \ \textrm{and} \ \ cd_0+1 = z_0^2,
\end{equation}
\begin{equation}\label{d1}
    bd_1+1 = y_1^2 \ \ \textrm{and} \ \ cd_1+1 = z_1^2.
\end{equation}

\begin{remark}
By \eqref{d0} and \eqref{d1}, if \(z_0\) or \(z_1\) is constant, then it is equal to \(\pm1\). As in \cite[Lemma 4]{dl_17}, if $(z_0,x_0)\neq (\pm 1,\pm 1)$, then $\deg(z_0) \geq\frac{\gamma}{2}$ and $\deg(x_0)\geq\frac{\alpha}{2}$. Analogously, if $(z_1,y_1)\neq (\pm 1,\pm 1)$, then $\deg(z_1) \geq\frac{\gamma}{2}$ and $\deg(y_1)\geq\frac{\beta}{2}$. 
\end{remark}

In order to extend a triple \( \{a, b, c\} \) to a Diophantine quadruple \( \{a, b, c, d\} \), we solve equations of the form
\begin{equation}
 z=v_m=w_n,\label{jedn-z}   
\end{equation} where binary recurrence sequences $(v_m)_{m\geq0}$ and $(w_n)_{n\geq0}$ are  given by
\begin{align}
    &v_0=z_0,\quad v_1=sz_0+cx_0,\quad v_{m+2}=2sv_{m+1}-v_m,\label{rekurzija_vm}\\
    &w_0=z_1,\quad w_1=tz_1+cy_1,\quad w_{n+2}=2tw_{n+1}-w_n.\label{rekurzija_wn}
\end{align}

\begin{remark}\label{predznaci}
 By taking $(\pm z_0, x_0)$ in \eqref{rekurzija_vm} we have covered all possible $z$'s, because all other pairs of the form $(\pm z_0,\pm x_0)$ can only lead to the same $d$. The same holds for $(\pm z_1, y_1)$ in \eqref{rekurzija_wn}.    
\end{remark}

The following lemmas are adapted variants of \cite[Lemma 5]{dl_17} and \cite[Lemma 1]{dij}, respectively. Given that the structure of the proof closely mirrors the original, we omit the details.

\begin{lem}\label{lema_stupnjevi}
Let $\{a,b,c,d\}$ be a Diophantine quadruple in $\mathbb{Q}(i)[X]$ and $(v_m)_{m\geq0}$ and $(w_n)_{n\geq0}$ be sequences defined as in (\ref{rekurzija_vm}) and (\ref{rekurzija_wn}), respectively. Then, for $m\geq 1$ and $n\geq 1$, the following holds:
\begin{equation}\label{degvm}
    \deg(v_m)=(m-1)\frac{\alpha+\gamma}{2}+\deg(v_1),
\end{equation}
\begin{equation}\label{degwn}
    \deg(w_n)=(n-1)\frac{\beta+\gamma}{2}+\deg(w_1).
\end{equation}
Also, 
\begin{equation}\label{degv1}
\frac{\gamma}{2}\leq \deg(v_1)\leq \frac{\alpha+5\gamma}{4},
\end{equation}
\begin{equation}\label{degw1}
    \frac{\gamma}{2}\leq \deg(w_1)\leq \frac{\beta+5\gamma}{4}.
\end{equation}
\end{lem}

\begin{lem}\label{degineq}
For $v_m = w_n$, where $(v_m)_{m\geq0}$ and $(w_n)_{n\geq0}$ are defined by (\ref{rekurzija_vm}) and (\ref{rekurzija_wn}), we have $n-1 \leq m \leq 2n+1.$
\end{lem}

We proceed by establishing several congruence relations for \( v_m \) and \( w_n \), which will be crucial in the proof of Theorem~\ref{glavni-tm}. Since the proof of the following lemma follows by a simple induction, we omit it.

\begin{lem}\label{initial1}
The sequences $(v_m)_{m\geq0}$ and $(w_n)_{n\geq0}$, given by (\ref{rekurzija_vm}) and (\ref{rekurzija_wn}), respectively, satisfy the following congruences
\begin{equation}\label{v2mc}
    v_{2m} \equiv v_0=z_0 \pmod{c}, \ \ v_{2m+1} \equiv v_1 \pmod{c},
\end{equation}
\begin{equation}\label{w2nc}
    w_{2n} \equiv w_0=z_1 \pmod{c}, \ \ w_{2n+1} \equiv w_1 \pmod{c}.
\end{equation}
\end{lem}

\section{Gap principles for degrees}
The following lemma is an adapted version of \cite[Lemma~2]{glavni} and, since the proof is analogous, it is omitted. In what follows, the sign in \(d_{\pm}\) is chosen so that \(d_-\) denotes the regular extension of smaller degree.

\begin{lem}\label{d-forms}
Let $\{a, b, c\}$ be a Diophantine triple in $\mathbb{Q}(i)[X]$ such that (\ref{jdbe_osnovne}) holds. Then, for 
\begin{equation}\label{dminus}
d_- = a+b+c+2(abc\pm rst),
\end{equation}
we have 
\begin{equation}\label{adbdcd}
ad_-+1 = u^2, \ \ bd_-+1 = v^2, \ \ cd_-+1 = w^2,
\end{equation}
where 
\begin{equation}\label{uvw-minus}
u = at \pm rs, \ \ v = bs \pm rt, \ \ w = cr \pm st.
\end{equation}
Additionally, we have
\begin{equation}\label{c-d_-}
c = a + b + d_- + 2(abd_-\pm ruv),
\end{equation}
\begin{equation}\label{d_-1}
d_- = a+b-c+2rw,
\end{equation}
\begin{equation}\label{d_-2}
d_- = a - b + c +2sv,
\end{equation}
\begin{equation}\label{d_-3}
d_- = -a+b+c+2tu.
\end{equation}
\end{lem}

The polynomial $d_-$ plays an important role in our analysis. The set $\{a,b,c,d_-\}$ is a regular Diophantine quadruple in $\mathbb{Q}(i)[X]$ and for $d_-\neq0$, from \cite[Lemma 1]{dfl} and \cite[Lemma 2]{dij}, we have 
\begin{equation}\label{ocjenad-}
0\leq \deg(d_-) \leq \gamma - \alpha - \beta.
\end{equation} This implies that $\gamma\geq \alpha+\beta$ but, by observing the possibilities for $d_-$, we conclude even more about polynomials $a$, $b$ and $c$ and possible relations between their degrees.

\begin{lem}\label{lemma-d_-}
Let $\{a, b, c\}$ be a Diophantine triple in $\mathbb{Q}(i)[X]$ and 
let \(d_-\) denote the regular extension of smaller degree defined in
\eqref{dminus}. Then
$$d_-=0 \ \ \text{or} \ \ \deg(d_-) = \gamma - \alpha - \beta.$$ 

If $d_-\neq 0$, then the following possibilities occur:
\begin{enumerate}
    \item $\alpha=0$, $\beta=\gamma$ and $d_-=a=\pm\frac{1}{2}\left(f - \frac{1}{f}\right)$, for $f\in \mathbb{Q}(i)\backslash\{0\}$. The special case \(a=\pm i\) is obtained by taking \(f=\pm i\). 
\item $d_- = c + sv + tu$, where $s,t$ are given with (\ref{jdbe_osnovne}) and $u,v$ with (\ref{adbdcd}), and
\begin{itemize}
    \item $\alpha=0$, $\beta<\gamma$ and $\deg(d_-)=\gamma-\beta>0$ or
    \item $\alpha>0$, $\gamma =\alpha+\beta$ and $\deg(d_-)=0$ or
    \item $\alpha>0$, $\gamma>\alpha+\beta$ and $\deg(d_-)>0$.
\end{itemize}
\end{enumerate}
\end{lem}

\begin{proof}
Let $d_-\neq 0$. We proceed as follows:
\begin{enumerate}
\item Assume first that $d_-$ is a constant polynomial and $\alpha =0$. Since we have at most one constant polynomial in a Diophantine quadruple in $\mathbb{Q}(i)[X]$, we conclude $d_- = a$. By (\ref{adbdcd}), we have $(a+u)(a-u)=-1$, for $a, u \in \mathbb{Q}(i)[X]$. This is possible for
\begin{enumerate}
\item $a = \pm i, \ u = 0$, which happens in the case of integer factorization, i.e. for $a, u \in \mathbb{Z}[i][X]$, 
and
\item $a+u = \pm f, \ \ a-u = \mp \frac{1}{f},$ when $f\in \mathbb{Q}(i)\backslash\{0\}$.
\end{enumerate}
We examine each case separately.
\begin{enumerate}
\item For $a = \pm i, \ u = 0$ by (\ref{d_-3}) we get $c = -b\pm 2i$, and obviously $\beta = \gamma$. 
\item Additionally, in this case for $f\in \mathbb{Q}(i)\backslash\{0\}$ we get
\[
(a, u) \in \left\{
\begin{array}{l}
\left( \frac{1}{2}\left(f - \frac{1}{f}\right), \ \frac{1}{2}\left(f + \frac{1}{f}\right) \right), \quad
\left( \frac{1}{2}\left(f - \frac{1}{f}\right), \ -\frac{1}{2}\left(f + \frac{1}{f}\right) \right), \\[5pt]
\left( \frac{1}{2}\left(\frac{1}{f} - f\right), \ -\frac{1}{2}\left(f + \frac{1}{f}\right) \right), \quad
\left( \frac{1}{2}\left(\frac{1}{f} - f\right), \ \frac{1}{2}\left(f + \frac{1}{f}\right) \right)
\end{array}
\right\}.
\]
From (\ref{d_-3}), we get 
\begin{equation}\label{degc-case}
c = 2a - b - 2tu = \left\{
\begin{array}{l}
f-\frac{1}{f}-b\pm t\left(f+\frac{1}{f}\right), \\[6pt]
\frac{1}{f}-f-b\pm t\left(f+\frac{1}{f}\right). \\[6pt]
\end{array}
\right.
\end{equation}
This situation leads to the conclusion that $\deg(c) = \max\left\{\beta, \frac{\beta+\gamma}{2}\right\}$, which implies $\beta = \gamma$. The special case \(a=\pm i\) is obtained from the above parametrization by taking \(f=\pm i\).
\end{enumerate}
\item We now assume that at least one of $\deg(d_-)$ and $\alpha$ is $>0$. In both cases, it follows from (\ref{ocjenad-}) that $\gamma > \beta$. By summing equations (\ref{d_-2}) and (\ref{d_-3}), we obtain
\begin{equation}\label{d-}
d_- = c + sv + tu.
\end{equation}
From equation (\ref{jdbe_osnovne}), we deduce that $\deg(s) = \frac{\alpha + \gamma}{2}$ and $\deg(t) = \frac{\beta + \gamma}{2}$, while equation (\ref{adbdcd}) yields $\deg(u) = \frac{\alpha + \deg(d_-)}{2}$ and $\deg(v) = \frac{\beta + \deg(d_-)}{2}$. Combining these expressions with the estimate in (\ref{ocjenad-}), we conclude that $\deg(sv),\allowbreak\ \deg(tu) \leq \gamma$. The case where both $\deg(sv)$ and $\deg(tu)$ are strictly less than $\gamma$ leads to $\deg(d_-) < \gamma - \alpha - \beta$, and hence to a contradiction, as \eqref{d-} would imply $\deg(c) < \gamma$. Therefore, by \eqref{ocjenad-}, we must have $\deg(d_-) = \gamma - \alpha - \beta$.
\end{enumerate}
\end{proof}

\begin{remark}\label{fto-f}
It is unnecessary to consider both cases
\( a = \pm \tfrac{1}{2}\left(f - \tfrac{1}{f}\right) \), $f\in \mathbb{Q}(i)\backslash\{0\}$, as transformation \( f \mapsto -f \) produces the alternative case, thus covering all possibilities.
\end{remark}

Lemma \ref{lemma-d_-} implies some possible forms of Diophantine triples in $\mathbb{Q}(i)[X]$, which occur if $\beta=\gamma$.

\begin{remark}\label{Rem_dminus}
If $d_{-}=0$, then $c=a+b+2r$, $s=\pm(a+r)$ and $t=\pm(b+r)$ or $c=a+b-2r$, $s=\pm(a-r)$ and $t=\pm(b-r)$.\par
\end{remark}

\begin{remark}\label{remark3}
In the special case of (1) in Lemma \ref{lemma-d_-}, we get $\{a, b, c\}=\{\pm i, b,-b\pm 2i\}$. Similarly to \cite[Remark 3]{glavni}, it follows that 
\begin{equation}\{a, b, c\}=\{\pm i, \pm ti\pm i, \mp ti\pm i\},\label{pr1}\end{equation}where $t=-r^2=s^2.$ The triple of the form (\ref{pr1}) can exist in $\mathbb{Z}[i][X]$.
\end{remark}

\begin{remark}\label{rem-leading-coefficients}
Capital letters will be used to denote leading coefficients of
the corresponding non-zero polynomials. Thus, if $a,b,c,r,s,t\in\mathbb Q(i)[X],$
then \(A,B,C,R,S,T\in\mathbb Q(i)\) denote their leading coefficients,
respectively. If \(d_-\neq0\), we denote by \(D_-\) the leading coefficient of
\(d_-\).
\end{remark}

\begin{remark}\label{remark5}
The general part of the case (1) in Lemma \ref{lemma-d_-}, since $d_- = a$, $ab+1 = r^2$ and $ab+1 = v^2$, leads to $\pm r = bs \pm rt$, so $bs = r(\mp t \pm 1).$
By (\ref{jdbe_osnovne}), $\gcd(b, r) = 1$, thus $r = ps$, for $p\in\mathbb{Q}(i)$ and
\begin{equation}\label{degb-case}
b = \mp p t \pm p = \pm p(1-t).
\end{equation}

From (\ref{degc-case}) we can conclude certain relations that hold for the leading coefficients of the polynomials, namely we conclude $C = -B \pm \sqrt{B C} \left(f + \frac{1}{f}\right),$ where $B$ and $C$ are the leading coefficients of the polynomials $b$ and $c$, respectively, and $\sqrt{B C}\in\mathbb{Q}(i)$. From (\ref{degb-case}), we have $B = \mp p \sqrt{B C}$, i.e. $\sqrt{B}=\mp p\sqrt{C}.$
Combining the last two equations and expressing $p$, we get
\begin{equation}\label{p-f}
p^2+p\left(f + \frac{1}{f}\right)+1=0,
\end{equation}
from which we conclude that $\left(f + \frac{1}{f}\right)^2-4=q^2,$ for $q\in\mathbb{Q}(i)$. So, in this case we get $q = \pm \left( f - \frac{1}{f} \right)=\pm 2a$ and then, by (\ref{adbdcd}), $f + \frac{1}{f} =\pm 2u.$ The roots of the quadratic equation (\ref{p-f}) are  $p = -f \ \text{and} \ p = -\frac{1}{f}.$ 
Let us note that $p=\pm a \mp u$, $f=\pm a\pm u$. By (\ref{degb-case}), we obtain
\begin{equation}\label{a-a}
b \in \left\{\mp f(1-t), \ \mp\dfrac{1}{f}(1-t)\right\}, \qquad
c \in \left\{\pm \dfrac{1}{f}(t+1), \ \pm f(t+1)\right\},
\end{equation}
for $p = -f $ and $p = -\tfrac{1}{f}$, respectively.
However, it is unnecessary to consider both corresponding families of solutions for \(b\) and \(c\),
as one can be derived from the other via the substitution \( f \mapsto \tfrac{1}{f} \). Furthermore, the choice of signs $\mp$ and $\pm$ is strictly determined by the sign of the first element of the triple. The upper signs correspond to the starting element $-a$, while the lower signs correspond to our starting element $a$. Consequently, the further analysis may, without loss of generality, be restricted to the family matching $a$ for $p = -f$, i.e.
\[
b = f(1-t), \quad c = -\dfrac{1}{f}(t+1).
\]
Hence, it suffices to restrict attention to
\begin{equation}\label{one-triple}
\{a,b,c\} = \Big\{ \tfrac{1}{2}\left(f - \tfrac{1}{f}\right), \ b, \ \tfrac{b}{f^{2}} - \tfrac{2}{f} \Big\},
\end{equation}with $f\in\mathbb{Q}(i)\backslash\{0\}$. As noted, the transformation $f \mapsto -1/f$ produces the alternative valid extension for the same element $a$, while $f \mapsto -f$ naturally produces extensions for $-a$, covering all possibilities.
If we choose our combination of the signs, for example $$\{a,b,c\}=\Big\{\frac{1}{2}\left(f - \frac{1}{f}\right), \ f(1-t), \ -\frac{1}{f}(t+1)\Big\},$$ using (\ref{jdbe_osnovne}) and $r=fs$
we obtain $r=\pm\frac{\sqrt{-f^2 (t-1)+t+1}}{\sqrt{2}}$, $s=\pm\frac{\sqrt{-f^2 (t-1)+t+1}}{f\sqrt{2}}$. By plugging $-f^2 (t-1)+t+1=2k^2$ into expressions for $r$ and $s$, we get $r=\pm k$ and $s=\pm \frac{k}{f}$. Now we can choose $f\in\mathbb{Q}(i)\backslash\{0\}$ and $k\in\mathbb Q(i)[X]$ to get a polynomial Diophantine triple in $\mathbb{Q}(i)[X]$. 

For $f=2i$ and $k=X$, we have $\{a,b,c\}=\Big\{\frac{5i}{4}, -\frac{4}{5}i(X^2-1), \frac{1}{5}i(X^2+4)\Big\},$ where $r=X$, $s=-\frac{iX}{2}$ and $t=\frac{1}{5}(2X^2+3)$. For suitable irrational \(f\), one obtains polynomial Diophantine triples in \(\mathbb C[X]\). It is interesting that a suitable triple of this form can be extended to an irregular Diophantine quadruple in $\mathbb{C}[X]$ (see \cite{dij}).
\end{remark}

\begin{remark}
Regular $D(1)$-quadruples of the form $\{a, b, c, d_{\pm}\}$, where $d_{\pm}$ is defined by equation~\eqref{jdba_d_pm}, can be improper. According to~\eqref{jdba_d_pm}, we have
\begin{equation}\label{d_+}
\deg(d_+) = \alpha + \beta + \gamma.
\end{equation}
Furthermore, improper $D(1)$-quadruples of the form $\{0, a, b, c\}$ always exist.
\end{remark}

Since any $D(1)$-triple $\{a, b, c\}$ can be extended to a regular $D(1)$-quadruple $\{a, b, c,d\}$, where $d=d_{\pm}$, equation (\ref{jedn-z}), with $(v_m)_{m\geq 0}$ given by (\ref{rekurzija_vm}) and $(w_n)_{n\geq 0}$ given by (\ref{rekurzija_wn}), always has solutions.
If $d=d_-$ then we can easily prove what values $m$ and $n$ can have. 

Assume now that $\{a, b, c, d'\}$, with $\deg(d') = \delta$ and $\gamma \leq \delta$, is an irregular $D(1)$-quadruple with minimal $\delta$ among all irregular $D(1)$-quadruples in $\mathbb{Q}(i)[X]$. According to Lemma \ref{lema_supanj_delta}, we have
\begin{equation}
    \delta \geq \frac{3\beta+5\gamma}{2}.
\end{equation}

The proof of the next lemma is completely analogous to the proof of \cite[Lemma 3.4]{fij}, adapted to this case, so the proof can be omitted. In the proof, the minimality assumption on $\delta$ is used.
\begin{lem}\label{gap0}
Let $\{a, b, c\}$ be a $D(1)$-triple in $\mathbb{Q}(i)[X]$. Let $v_m=w_n$ and let $d = \dfrac{v_m^2-1}{c}$.
\begin{itemize}
    \item If $d = d_{-}$, then $m, n \in \{0, 1\}$. 
    \item If $d=d'$, then $m\geq 3$ and $n\geq 3$.
\end{itemize}
Moreover, the following holds: 
\begin{itemize}
    \item If $0 \in \{m, n\}$ then  $d=d_-$ or $d = 0 \neq 
 d_-$ or $d = a \neq d_-$.
 \item If $(m, n) = (1, 1)$, then $d=d_-$ or $d = 0 \neq d_-$ or $d = a \neq d_-$ or $d = d_+$ and $\gamma \geq \alpha + 2\beta$.
\end{itemize}
\end{lem}

The following lemma gives a more precise description of all the possible relations between degrees of the polynomials in $D(1)$-triple. Let us note that conclusions about degrees and the leading coefficients generally hold, but conclusions about initial terms hold only for equations $v_m=w_n$ from which the solution $d = d_{-}$ arises. This lemma is an     appropriate analogue of \cite[Lemma 3.5]{fij} and \cite[Lemma 5]{glavni}. One difference from $\mathbb{R}[X]$ is that in $\mathbb{Q}(i)[X]$ we are not able to precisely determine the signs $\pm$, so all combinations of the signs are possible. Also, in $\mathbb{R}[X]$ we have more information about $d_0$ and consequently about $z_0$ and $x_0$. On the other hand, in $\mathbb{Z}[i][X]$, some of the possibilities for \(d_-\) occurring here do not arise.

\begin{lem}\label{od-}
Let $\{a, b, c\}$ be a $D(1)$-triple in $\mathbb{Q}(i)[X]$. The following cases can occur:
\begin{enumerate}
    \item If $d_-=0$, then $z_0 = z_1 = \pm 1$. In this case $c = a + b \pm 2r$ and $\beta = \gamma$. If $\alpha < \beta$, then $C = B$ and, if $\alpha = \beta$, then $C = A + B \pm 2\sqrt{AB}$.
    \item For $\deg(d_-) = 0$, one of the following options can occur:
    \begin{enumerate}
   \item  If $d_- = a$, then $(z_0, z_1) = (\pm s, \pm s)$, $\alpha = 0,$  and  $\beta = \gamma$.
    \item If $d_- \in \mathbb{Q}(i)\backslash\{0, a\}$, then $z_0 = z_1 = \pm cr \pm st$, $\alpha > 0$, $\gamma = \alpha + \beta$ and $C=4ABD_-$.
    \end{enumerate}    
     \item If $\deg(d_-)>0$, then $C=4ABD_-$ and the following possibilities occur:
    \begin{enumerate}
    \item $z_0 = z_1 = \pm cr \pm st$, with $\deg(d_-) \leq \alpha, \ \alpha > 0$ and $\alpha + \beta < \gamma \leq 2\alpha + \beta$,
    \item $(z_0, z_1) = (\pm cr \pm st, \pm s),$ where $\alpha \leq \deg(d_-) \leq \beta$, $\alpha \geq 0$ and $2\alpha + \beta \leq \gamma \leq \alpha + 2\beta,$
    \item $(z_0, z_1) = (\pm t, \pm cr \pm st)$, with $\deg(d_-) = \alpha$, $\alpha = \beta$ and $\gamma = 3\alpha$,
    \item $(z_0, z_1) = (\pm t, \pm s)$, where $\beta \leq \deg(d_-) < \gamma$, $\alpha \geq 0$ and $\gamma \geq \alpha + 2\beta$.
    \end{enumerate}
\end{enumerate}
\end{lem}
\begin{proof}
\begin{enumerate}
    \item For $d_-=0$, we get $u,v,w \in \{\pm 1\}$ from (\ref{adbdcd}), which leads us to the conclusion that $c = a + b \pm 2r$ according to (\ref{c-d_-}). Additionally, from $\gamma \leq \beta$ we get $\gamma = \beta$.
    \begin{itemize}
    \item If $\alpha < \beta$, then $C = B$.
    \item If $\alpha = \beta$, then $C = A + B \pm 2\sqrt{AB}$.
    \end{itemize}
    \item 
    \begin{enumerate}
    \item Let $d_- = a \in\mathbb{Q}(i)\backslash\{0\}$. In this case, $\alpha = 0$ and $\beta = \gamma$. By Lemma \ref{lemma-d_-} and Remark \ref{remark5}, 
    $a = \frac{1}{2}\left(f-\frac{1}{f}\right)$, $f \in \mathbb{Q}(i)\backslash\{0\},$  while $c = \pm \frac{1}{f}(t+1)$. 
    From \(cd_-+1=w^2\), we obtain \(w=\pm s\). Hence, by Lemma~\ref{gap0}, the value \(d_-\) can arise only from an equality
\(v_m=w_n=\pm s\), with \(m,n\in\{0,1\}\).
    In case $(m, n) = (0, 0)$, from (\ref{rekurzija_vm}) and (\ref{rekurzija_wn}), we get $z_0 = z_1 = \pm s$. For $(m, n) = (0, 1)$, from (\ref{rekurzija_vm}) and (\ref{rekurzija_wn}), we get 
    \begin{equation}\label{44}
        z_0 = tz_1 + cy_1 = \pm s.
    \end{equation}
     From (\ref{uvw-minus}) and (\ref{44}), we get $\pm cr \pm st = tz_1 + cy_1 = \pm s,$ so
        \begin{equation}\label{lemma2a1}
    c(\pm r - y_1) = t(z_1 \mp s).
    \end{equation}
    Since $(c,t)=1$, we have $t \mid (\pm r - y_1)$. Considering degrees of the polynomials we conclude $y_1 = \pm r$ and $z_1 = \pm s$. For $(m, n) = (1, 0)$ we get $sz_0 + cx_0 = z_1 = \pm s$, or $cx_0 = s(\pm 1 - z_0)$. We also have $\pm cr \pm st = sz_0 + cx_0$, i.e. $c(\pm r-x_0)=s(z_0\pm t)$ and, by Remark \ref{remark5}, $s\mid r$. Since $(c,s)=1$, it follows that $s\mid x_0$, which is possible only for $x_0=0$. For $x_0 = 0$, we get $z_0 = \pm 1$, which is a contradiction. In case $(m, n) = (1, 1)$, we get $sz_0 + cx_0 = tz_1 + cy_1 = \pm s$ which again leads to a contradiction.
    \item Now, $d_-\in\mathbb{Q}(i)\backslash\{0, a\}$. By Lemma \ref{lemma-d_-}, $\alpha>0$ and $\gamma = \alpha+\beta$. From $cd_- + 1 = w^2$, we conclude $w = \pm cr \pm st \neq \pm s$. By Lemma \ref{gap0}, we get $v_m = w_n = w$ for $m, n \in \{0, 1\}$. In case $(m, n) = (0, 0)$, according to (\ref{rekurzija_vm}) and (\ref{rekurzija_wn}), we get $z_0 = z_1 = \pm cr \pm st$. In cases $(m, n) \in \{(0, 1), (1, 0), (1, 1)\}$ we obtain a contradiction by arguments analogous to those used in Case 2.(a). By (\ref{adbdcd}) and Lemma \ref{lemma-d_-}, we have $\deg(c)=\deg(sv)=\deg(tu)=\gamma$ and we can also conclude that $C=4ABD_-$, where capital letters denote the leading coefficients of the corresponding polynomials.
    \end{enumerate}
    \item In case $\deg(d_-) > 0$, by Lemma \ref{lemma-d_-} we conclude $\gamma > \alpha + \beta$. From Lemma \ref{gap0} we get that if $d=d_-$, then $v_m = w_n = \pm w$ with $m, n \in \{0, 1\}$, so $d_-$ occurs from $v_m = w_n = \pm cr \pm st$ for $m, n \in \{0, 1\}$. Additionally, by Lemma \ref{lemma-d_-} and $cd_{-}+1= w_{-}^2$, we obtain
\begin{equation}\label{degw-}
\deg(w_{-}) = \gamma - \frac{\alpha+\beta}{2} < \gamma.
\end{equation} Similarly to case 2.(b), we conclude $C=4ABD_-$.
    \begin{enumerate}    
    \item If $(m, n) = (0, 0)$, we have $z_0 = z_1 = \pm cr \pm st$. By (\ref{degw-}) and (\ref{ineq20}) we get $\gamma \leq 2\alpha + \beta$. Moreover, by Lemma \ref{lemma-d_-} we obtain $\deg(d_-) \leq \alpha$, so $\alpha > 0$.
    \item If $(m, n) = (0, 1)$, then we have $z_0 = tz_1 + cy_1 = \pm cr \pm st$, so we obtain (\ref{lemma2a1}) and conclude $t \mid (\pm r - y_1)$. As before, $y_1 = \pm r$ and $z_1 = \pm s$. By (\ref{degw-}) and (\ref{deg1}) we get $\gamma \leq \alpha + 2\beta$. So, Lemma \ref{lemma-d_-} yields the inequality $\deg(d_-) \leq \beta$. Using (\ref{deg1}) and (\ref{jdbe_osnovne}) we obtain $\gamma \geq 2\alpha + \beta$. Again, by Lemma \ref{lemma-d_-}, $\deg(d_-)\geq \alpha$.
    \item If $(m, n) = (1, 0)$, similarly to the previous case, we get $x_0 = \pm r, \ z_0 = \pm t$ and $z_1 = \pm cr \pm st$. Additionally, we obtain $\alpha + 2\beta \leq \gamma \leq 2\alpha + \beta$, from which we get $\alpha = \beta$ and consequently $\gamma = 3\alpha$.
    \item For $(m, n) = (1, 1)$, using results from cases before, we obtain $z_0 = \pm t$ and $z_1 = \pm s$. Conclusion about degrees follows from (\ref{deg1}) and (\ref{ineq20}). Lemma \ref{lemma-d_-} leads us to the conclusion $\deg(d_-)\geq \beta$.
    \end{enumerate}
\end{enumerate}
\end{proof}

The following lemma is motivated by observing \cite[Lemma 10]{dif2} for $\mathbb{Q}(i)[X]$.
\begin{lem}
\label{Lm-ubaceno_novo}Let $\{a,b,c\}$ be a $D(1)$-triple in $\mathbb{Q}(i)[X]$ with $\beta< \gamma= \alpha+2 \beta.$ Then the triple \(\{a,b,d_-\}\) is either of the form \eqref{one-triple} or is regular, that is, $d_-=a+b\pm2r.$
\end{lem}

\begin{proof}
 For the $D(1)$-triple $\{a,b,c\}$, by Lemma \ref{lemma-d_-}, we conclude $\rm{deg}(\textit{d}_{-})=\beta$. Hence, the $D(1)$-triple $\{a,b,d_{-}\}$ is either regular or it has the form (1) from Lemma \ref{lemma-d_-}.
 
For the regular triple $\{a,b,d_{-}\}$, by Definition \ref{expc}, $d_{-}=a+b\pm2r$ holds. Similarly to \cite[Lemma 10]{dif2}, we get $c=4r(a\pm r)(b\pm r)$.

If the triple $\{a,b,d_{-}\}$ has the form (1) from Lemma \ref{lemma-d_-}, by Remark \ref{remark5}, we have $a=\frac{1}{2}\Big(f-\frac{1}{f}\Big)$, where $f\in\mathbb{Q}(i)\backslash\{0\}$, $b = f(1-v)$ and $d_- = -\frac{1}{f}(v+1)$. By simple algebraic operations, from these expressions, it follows $d_- =\frac{1}{f^2}b - \frac{2}{f}$. We conclude that $u=-\frac{r}{f}$ and $v=1-\frac{b}{f}$. Using (\ref{c-d_-}), an expression $c=c_1b^2+c_2b+c_3$, with $c_i\in\mathbb{Q}(i)$ for $i=1,2,3$, can be obtained for a certain value of $f$.
\end{proof} 
 
By Lemma \ref{lemma-d_-} and Lemma \ref{od-}, we  directly obtain the following result.

\begin{lem}\label{lemma-3.7-param}
  Let $\{a, b, c\}$ be a $D(1)$-triple in $\mathbb{Q}(i)[X]$ with $\beta < \gamma = 2\beta$. Assume that the constant \(a\) is given by $a = \frac{1}{2}\left(f - \frac{1}{f}\right)$ for some $f \in \mathbb{Q}(i) \setminus \{0, \pm 1\}$. In this case, $\{a, b, d_-, c\}$ is one of the following:
  \begin{enumerate}
      \item $\{a, \ b, \ a+b+ 2r, \ 4r(a+ r)(b+ r)\}$ and $s=\pm(2a(b+r)+1)$, $t=\pm (2b(a+r)+1)$,
      \item $\{a, \ b, \ a+b-2r,  \ 4r(r- a)(b- r)\}$ and $s=\pm(2a(b-r)+1)$, $t=\pm(2b(r-a)-1)$,
      \item $\left\{a, \ b, \ \frac{b}{f^2} - \frac{2}{f}, \ \frac{2(f^2-1)}{f^3}b^2 + 2\left(\frac{3}{f^2}-1\right)b + \frac{f^2-9}{2f}\right\}$, with the corresponding roots $s = \pm \left( b\frac{1-f^2}{f^2} + \frac{f^2-3}{2f} \right)=\pm(-2\frac{r^2}{f}+\frac{1}{f}+a)$ and $t = \pm r\left(1 - \frac{2b}{f}\right)$.
  \end{enumerate}
\end{lem}

\section{All possible fundamental solutions}
The main goal of this section is to determine all possible fundamental solutions $(z_0,x_0)$ and $(z_1,y_1)$ of equations (\ref{jdba_pellova_prva}) and (\ref{jdba_pellova_druga}), respectively. The next lemma gives us the relations between the initial terms $z_0, z_1, x_0, y_1 \in\mathbb{Q}(i)[X]$. Since the proof is completely analogous to \cite[Lemma 3]{dij}, and to its more precise version \cite[Lemma 7]{fij} for $\mathbb{R}[X]$, the proof is omitted here. The only difference is that over \(\mathbb{Q}(i)[X]\) we cannot obtain
precise conclusions about the signs \(\pm\), as is possible over
\(\mathbb{R}[X]\).

\begin{lem}\label{congr_1}\hspace{0pt}
\begin{enumerate}
    \item If $v_{2m}=w_{2n}$, then $z_0 = z_1$. 
    \item If $v_{2m+1} = w_{2n}$, then either $(z_0, z_1) = (\pm 1, \pm s)$ or $(z_0, z_1) = (\pm s, \pm 1)$ or $z_1 = cx_0+sz_0$, or $z_1 = cx_0-sz_0$, where $x_0$ is not a constant.
    \item If $v_{2m} = w_{2n+1}$, then either $(z_0, z_1) = (\pm t, \pm1)$ or $(z_0, z_1) = (\pm s, \pm 1)$ or $(z_0, z_1) = (\pm 1, \pm 1)$ or $z_0 = tz_1 + cy_1$ or $z_0 = tz_1-cy_1,$ where $y_1$ is not a constant.
    \item If $v_{2m+1} = w_{2n+1}$, then either $(z_0, z_1) = (\pm 1, \pm cr\pm st )$ or $(z_0, z_1) = (\pm cr \pm st, \pm 1) $ or $sz_0 + cx_0 = tz_1 \pm cy_1$ or $sz_0 - cx_0 = tz_1 \pm cy_1$, where $x_0, y_1$ are not constant polynomials and polynomials on both sides of the equation have degree less than $\gamma$.
\end{enumerate}
\end{lem}

The following consequence of Lemma~\ref{congr_1} corresponds to \cite[Lemma 4]{dij}.
\begin{lem}\label{alpha0}
    If $\alpha = 0$ and $x_0$ is a constant, then  $x_0^2 = a^2 +1$ and $z_0 = \pm s$. 
\end{lem}

\begin{remark}\label{constant}
 The situation from Lemma \ref{alpha0} can be described more precisely. By (\ref{d0}), $d_0=a$ and $x_0^2 = a^2 +1$. Hence, $a,x_0\in\mathbb{Q}(i)$. Analogously to the proof of Lemma \ref{lemma-d_-}, we get $a= \pm\frac{1}{2}\big (f-\frac{1}{f}\big ),$ $x_0 = \pm \frac{1}{2}\big (f+\frac{1}{f}\big ),$ where $f\in \mathbb{Q}(i)\backslash\{0\}$. The transformation \(f\mapsto -f\) replaces \(a\) by \(-a\), so it is enough
to consider $a=\frac12\left(f-\frac1f\right).$
Also, by Remark \ref{predznaci}, it is enough to observe $x_0 = \frac{1}{2}\big (f+\frac{1}{f}\big ).$
\end{remark}

Now, we examine Lemma \ref{congr_1} and determine initial terms which appear in $\mathbb{Q}(i)[X]$ from these possibilities. In the following lemma for each possibility of initial terms we have relations between degrees $\alpha$, $\beta$ and $\gamma$ which admit that possibility. The case 1.(a) is always possible, because the $D(1)$-triple $\{a,b,c\}$ can always be extended to a regular or irregular $D(1)$-quadruple by adding $0$ to that set. For one particular $D(1)$-triple $\{a, b, c\}$ there can be more initial terms, depending on degrees of polynomials $a$, $b$ and $c$. The cases 2.(a), 2.(b), 3.(a), 3.(c), 4.(a) and 4.(b) are not possible in $\mathbb{Z}[i][X]$ (see \cite[Lemma 9]{glavni}), but they are possible in $\mathbb{R}[X]$ (see \cite[Lemma 4.3]{fij}) with the difference that in $\mathbb{R}[X]$ we were able to determine the signs $\pm$. Since we have to examine some possibilities which do not occur in
\(\mathbb Z[i][X]\), and since over \(\mathbb Q(i)[X]\) the signs cannot be fixed as in \(\mathbb R[X]\), we give the full proof. It is interesting to note that in $\mathbb{C}[X]$ from the case 4.(b)ii. of the following lemma an irregular polynomial $D(1)$-quadruple $\mathcal{D}_p$ arises (see \cite[Proposition 1]{dij}).

\begin{lem}\label{fund_cases}
\begin{enumerate}
\item If $v_{2m} = w_{2n}$, then either
\begin{enumerate}
\item $z_0 = z_1 = \pm 1$ or
\item  $z_0 = z_1 = \pm s $ and $\alpha = 0$ or
\item $z_0 = z_1 = \pm cr\pm st$ and $\alpha >0$, \ $\alpha + \beta \leq \gamma \leq 2\alpha + \beta$,
\end{enumerate}
\item If $v_{2m+1} = w_{2n}$, then either
\begin{enumerate}
\item $(z_0, z_1) = (\pm 1, \pm s)$ and $\gamma \geq 2\alpha + \beta$ or 
\item $(z_0, z_1) = (\pm s,\pm 1)$ and $\alpha = 0, \ \beta \leq \gamma$ or
\item $(z_0, z_1) = (\pm t, \pm cr \pm st)$, $\alpha = \beta, \ \gamma = 3\alpha$.
\end{enumerate}
\item If $v_{2m} = w_{2n+1}$, then
\begin{enumerate}
\item $(z_0, z_1) = (\pm t, \pm 1)$ and $\gamma \geq \alpha + 2\beta$ or 
\item $(z_0, z_1) = (\pm s, \pm 1), \ \alpha = 0, \ \beta = \gamma$ or
\item $(z_0, z_1) = (\pm 1, \pm 1)$,  $\alpha = 0, \ \beta = \gamma$ or
\item $(z_0, z_1) = (\pm cr \pm st, \pm s)$ and $\alpha \geq 0, \ 2\alpha + \beta \leq \gamma \leq \alpha + 2\beta$ (special case:
\begin{enumerate}
    \item $(z_0, z_1) = (\pm s, \pm s)$ and $\alpha = 0, \ \beta = \gamma$).
\end{enumerate}
\end{enumerate}
\item If $v_{2m+1} = w_{2n+1}$, the following options can occur:
\begin{enumerate}
\item $(z_0, z_1) = (\pm 1, \pm cr\pm st)$ and $\gamma \leq 2\alpha + \beta$ (special cases: 
\begin{enumerate}
    \item $(z_0,z_1) = (\pm 1,\pm 1)$ and $\alpha \leq \beta = \gamma$ and 
    \item $(z_0,z_1) = (\pm 1,\pm s)$ and $\alpha = 0$, $\beta = \gamma$) or
\end{enumerate} 
\item $(z_0, z_1) = (\pm cr\pm st, \pm 1)$ and $\gamma \leq \alpha + 2\beta$
(special cases: 
\begin{enumerate}
    \item 4.(a)i. and
    \item $(z_0,z_1) = (\pm s,\pm 1)$ and $\alpha = 0$, $\beta = \gamma$) or
\end{enumerate}
\item $(z_0, z_1) = (\pm t, \pm s)$ and $\gamma \geq \alpha + 2\beta$.
\end{enumerate}
\end{enumerate}
\end{lem}
\begin{proof}
\begin{enumerate}
\item From Lemma \ref{congr_1} we know $z_0 = z_1$. From the relations (\ref{rekurzija_vm}) and (\ref{rekurzija_wn}), we get $v_0 = w_0$, so we apply Lemma \ref{gap0}. The cases where $d = d_-$ are described in the parts 1.-3.(a) of Lemma \ref{od-}. If $d=0 \neq d_-$, then $z_0 = z_1 = \pm 1$. By Lemma \ref{lemma-d_-}, if $\beta = \gamma$ then $d_-=a$. Otherwise, $\beta < \gamma$. If $d = a \neq d_-$ and $\alpha = 0$, then $z_0 = z_1 = \pm s$. By Lemma \ref{lemma-d_-}, if $\beta = \gamma$ we get $d_-=0$.
\item One of the possibilities from Lemma \ref{congr_1} is $(z_0, z_1) = (\pm 1, \pm s)$. In this case $x_0 = \pm 1$, so Lemma \ref{initial1} leads to congruence $\pm s \pm c \equiv \pm s \pmod{c}$. Hence, we have either $\pm c \equiv 0 \pmod{c}$ which is true or $\pm c \equiv \pm 2s \pmod{c}$ which is a contradiction to (\ref{jdbe_osnovne}). By (\ref{ineq20}), we get $\gamma \geq 2\alpha + \beta$. \\
Now, we assume $(z_0, z_1) = (\pm s, \pm 1)$. By \eqref{d0}, $d_0=a$ and $x_0^2=a^2+1$. This is possible only if $a,x_0\in\mathbb{Q}(i)$. By Lemma \ref{initial1}, we get $\pm s^2+cx_0 \equiv \pm 1 \pmod{c}$, and $(\pm a+x_0)c \equiv 0 \pmod{c}$ holds. We have $\alpha = 0$ and we only know that $\gamma\geq\beta$. Possible relations between $\beta$ and $\gamma$, depending on the form of $d_-$, are described in Lemma \ref{od-}.\\
The last possibilities for this setup occur for $z_1 = cx_0 \pm sz_0$, where $x_0$ is not a constant. We have $\pm v_1 = w_0$. So, by Lemma \ref{gap0}, we can have $d_1 = d_-$ or $d_1 = 0 \neq d_-$ or $d_1 = a \neq d_-$. We know that $d_-$ occurs from $v_m = w_n = \pm cr \pm st$ for $m, n \in \{0, 1\}$ and, for $d_- \neq 0$,  $\deg(d_-)=\gamma -\alpha-\beta$. The situation when $d_1 = d_-$ is described in 3.(c) of Lemma \ref{od-}. For $d_1=0\neq d_-$, $sz_0\pm cx_0=\pm 1$. The equality $\alpha = \gamma$ leads to $\gamma = 0$, and the contradiction is obtained, so we conclude $\alpha < \gamma$. Moreover, using (\ref{jdba_pellova_prva}), we get
\begin{equation}\label{eq1}
(sz_0+cx_0)(sz_0-cx_0) = z_0^2 +ac-c^2.
\end{equation}
From (\ref{eq1}), we obtain $\deg(sz_0\pm cx_0) = 2\gamma$, which is a contradiction to (\ref{deg1}).  If $d_1 = a \neq d_-$, then $sz_0\pm cx_0 = \pm s$. So, $s \mid x_0$, which is not possible for $x_0 \neq 0$ because of (\ref{deg1}). So, for $x_0 = 0$, $d_1 = a$. In this case, $z_0 = \pm s$, and $\pm s^2  = \pm s$, a contradiction.
\item The first possible option from Lemma \ref{congr_1} is $(z_0, z_1) = (\pm t, \pm 1)$. In this case $y_1 = \pm 1$, so Lemma \ref{initial1} leads to the congruence $\pm t \equiv \pm t \pm c \pmod{c}$. Then either $\pm c \equiv 0 \pmod{c}$, which is true, or $\pm c \equiv \pm 2t \pmod{c}$ which contradicts the initial equations (\ref{jdbe_osnovne}). The conclusion about degrees follows from (\ref{deg1}).\\
Another option we observe is $(z_0, z_1) = (\pm s, \pm 1)$. In this case, we again have $y_1 = \pm 1$ and from Lemma \ref{initial1} the congruence $\pm s \equiv \pm t \pm c \pmod{c}$ is obtained. We have $d_0 = a$ and, by (\ref{d0}), $x_0^2=a^2+1$. This is possible only for $a\in\mathbb{Q}(i)$. Since $c\mid \pm s \pm t$, we conclude that $\beta = \gamma$.\\
The next case that can be valid, according to Lemma \ref{congr_1}, is $(z_0, z_1) = (\pm 1, \pm 1)$. Then again $y_1 = \pm 1$ and, from Lemma \ref{initial1}, we get the congruence
\begin{equation}\label{congr3_1}
\pm 1 \equiv \pm t \pm c \pmod{c}.
\end{equation}
By (\ref{congr3_1}), $c\mid \pm 1\pm t$ so $\beta=\gamma.$ Lemma \ref{lemma-d_-} implies that $d_-=0$ or $d_-=a$ but, according to the proof of \cite[Lemma 7]{fij}, in this case $d_- = a$. 
For $a = \frac{1}{2}\left(f - \frac{1}{f}\right)$, by Remark \ref{remark5}, we restrict our attention to the family where $c = -\frac{1}{f}(t+1)$, which implies $t = -cf - 1$. So, (\ref{congr3_1}) becomes:
\begin{equation}
\pm 1 \equiv \mp cf \mp 1 \pm c\pmod{c}.
\end{equation}
This is possible in $\mathbb{Q}(i)[X]$ for suitable combinations of the independent signs (yielding $1 \equiv 1$ or $-1 \equiv -1$).
Finally, there is also a possibility when $z_0 = tz_1 \pm cy_1$, where $y_1$ is not a constant polynomial. Obviously, we have $v_0 = \pm w_1$ so, by Lemma \ref{gap0}, we get $d_0 = d_-$ or $d_0 = 0 \neq d_-$ or $d_0 = a \neq d_-$. The cases when $d_- = d_0$ are described in 2.(a) and 3.(b) of Lemma \ref{od-}. If $d_0 = 0\neq d_-$, then $tz_1 \pm cy_1 = \pm 1$. If $\beta = \gamma$, we conclude $d_- = a$ and $\alpha=0$. We get
\begin{equation}\label{eq2}
(cy_1 + tz_1)(cy_1 - tz_1) = c^2 - bc - z_1^2.
\end{equation}
We have $\deg(tz_1 \mp cy_1) = 2\gamma$. If $\beta < \gamma$  we get $\deg(cy_1 \mp tz_1) = 2\gamma$ which is not possible because of (\ref{ineq20}). If $d_0 = a \neq d_-$, then $tz_1 \pm cy_1 = \pm s$. If $\beta < \gamma$, then by (\ref{eq2}) we get $\deg(tz_1 \pm cy_1) = \frac{3\gamma}{2}$, a contradiction with (\ref{ineq20}). For $\beta = \gamma$, by Lemma \ref{lemma-d_-}, we get $d_- = 0$, so $c = a + b \pm 2r$ and, by (\ref{eq2}), we get $\deg(tz_1 \mp cy_1) = \gamma$. We have $tz_1 \equiv \pm s \pmod{c}$. Multiplying that by $t$, we obtain $z_1 \equiv \pm st \mp cr \pmod{c}$. Since
\begin{equation}\label{eq3}
    (\pm st - cr)(\pm st + cr) = ac + bc + 1 - c^2,
\end{equation}
one of the polynomials $\pm st \pm cr$ has a degree less than $\gamma$, and the other has degree $\gamma + \frac{\alpha+\beta}{2}$. Hence, $\pm st \pm cr = z_1$. From (\ref{eq3}) and $c = a+b\pm 2r$, we get $\deg(z_1)=0$. So, $z_1 = \pm 1$ and $y_1 = \pm 1$, which is not possible.
\item From Lemma \ref{congr_1}, one possibility is $(z_0, z_1) = (\pm 1, \pm cr \pm st) = (\pm 1, \pm w)$. So, $x_0 = \pm 1$ and $y_1 = \pm v$ and, according to Lemma \ref{initial1}, we get $ \pm s \pm c \equiv \pm crt \pm s \pm bcs \pm bcs \pm crt \pmod{c}.$ The option $c \mid s$ is not possible because of (\ref{jdbe_osnovne}), but $\pm c \equiv 0 \pmod{c}$ is true for suitable signs $\pm$ in congruence. By (\ref{ineq20}), we get the bound for $\gamma$. In particular, for $d_- = 0$, by Lemma \ref{d-forms}, $z_1 = \pm 1$ and $\alpha \leq \beta = \gamma$. For $d_- = a$, we have $z_1 = \pm s$ and $\alpha = 0$, $\beta = \gamma$.\\
We now consider the case where $(z_0, z_1) = (\pm cr \pm st, \pm 1)$. Similarly to the previous case, either $c\mid t$ which is not possible or $c\equiv 0 \pmod{2c}$ which is possible in $\mathbb{Q}(i)[X]$ for the suitable signs $\pm$. Conclusions about degrees and special cases are analogous to in the previous case. \\
Finally, we have $sz_0 \pm cx_0 = tz_1 \pm cy_1$, where $x_0, y_1$ are not constants, and polynomials on both sides of the equation have degrees less than $\gamma$. So, we have $\pm v_1 = \pm w_1$. By Lemma \ref{gap0}, this can lead to $d_-$ or $d_+$, if $\gamma \geq \alpha + 2\beta$, or to an irregular $D(1)$-quadruple $\{a, b, c, d\}$ where $d = 0\neq d_-$ or $d = a \neq d_-$. Cases where we obtain $d_+$ are possible \cite[Proposition 1]{dij} only if equations $sz_0 \pm cx_0 = tz_1 \pm cy_1$ hold for both signs $\pm$. Then, we have $(z_0,z_1)=(\pm t,\pm s)$. If $d=0\neq d_-$ or $d = a \neq d_-$, then $sz_0 \pm cx_0 = \pm 1$ or $sz_0 \pm cx_0 = \pm s$, respectively. These situations are not possible as it is shown in part (2).
\end{enumerate}
\end{proof}

The following lemma allows us to reduce the number of cases in Lemma \ref{fund_cases}, because in some cases we actually have the same intersections of sequences, but with shifted indices. It is a version for $\mathbb{Q}(i)[X]$ of \cite[Lemma 2.3]{cfm} for integers. The proof is analogous to the proof of \cite[Lemma 10]{glavni} for $\mathbb{Z}[i][X]$.

\begin{lem}\label{preklapanja}Let $v_{z_0,m}$ be the $m$-th term of the sequence $(v_m)_{m\geq 0}$, given by (\ref{rekurzija_vm}), with initial term $z_0$ and $w_{z_1,n}$ the $n$-th term of the sequence $(w_n)_{n\geq 0}$, given by (\ref{rekurzija_wn}), with initial term $z_1$. Then
$$\begin{tabular}{l l l }
$v_{t,m}=-v_{cr-st,m+1},$ & $v_{-t,m+1}=-v_{-cr+st,m},$ \\
$v_{t,m+1}=v_{cr+st,m},$ & $v_{-t,m}=v_{-cr-st,m+1},$ \\
$w_{s,n}=-w_{cr-st,n+1},$ & $w_{-s,n+1}=-w_{-cr+st,n},$ \\
$w_{s,n+1}=w_{cr+st,n},$ & $w_{-s,n}=w_{-cr-st,n+1}.$
\end{tabular}$$
\end{lem}

By Lemma \ref{preklapanja}, some cases of Lemma \ref{fund_cases} can be reduced to the other ones. More precisely, after a possible change of signs and a shift of one of the indices, the same intersections are obtained. However, the degree conditions are inherited from the original cases. Therefore, whenever such a reduction is used in the proof of Theorem \ref{glavni-tm}, we reduce only the form of the intersection, while the degree conditions remain those of the original case.

\section{Proof of the main theorem}
In order to check all possible extensions of an arbitrary $D(1)$-triple $\{a,b,c\}$ to a $D(1)$-quadruple $\{a,b,c,d\}$ in $\mathbb{Q}(i)[X]$, we search for suitable solutions of equation (\ref{jedn-z}), where $(v_{m})_{m\geq 0}$ and $(w_{n})_{n\geq 0}$ are binary recurrence sequences defined by (\ref{rekurzija_vm}) and (\ref{rekurzija_wn}), for some initial values $(z_{0},x_{0})$ and $(z_{1},y_{1})$. In Lemma \ref{fund_cases} all possible initial terms are described. In the proof of the theorem, the majority of cases from Lemma \ref{fund_cases} will be addressed, while those not explicitly treated are covered by previously established cases from the same lemma, based on the results of Lemma \ref{preklapanja}. The aim of this section is to prove that none of these cases leads to an irregular \(D(1)\)-quadruple with \(d=d'\). We use relations in $\mathbb{Q}(i)[X]$ from \cite[Lemma 6]{dif2}:

\begin{lem}\label{Lm8}Let the sequences $(v_{m})_{m\geq 0}$ and $(w_{n})_{n\geq 0}$ be given by  (\ref{rekurzija_vm}) and (\ref{rekurzija_wn}). Then, \begin{eqnarray*}
v_{2m}&\equiv&z_{0}+2c(az_{0}m^2+sx_{0}m)\ ({\rm mod}\ {c^2}), \\
v_{2m+1}&\equiv&s z_{0}+c[2asz_{0}m(m+1)+x_{0}(2m+1)]\ ({\rm mod}\ {c^{2}}),\\
w_{2n}&\equiv&z_{1}+2c(bz_{1}n^{2}+ty_{1}n)\ ({\rm mod}\ {c^{2}}), \\
w_{2n+1}&\equiv&t z_{1}+c[2btz_{1}n(n+1)+y_{1}(2n+1)]\ ({\rm mod}\ {c^{2}}).
\end{eqnarray*}\end{lem} 

We also use the following result \cite[Lemma 5.2]{fij}, adjusted for $\mathbb{Q}(i)[X]$:
\begin{lem}\label{Lm-dodano}Let $\{a,b,c\}$ be a possibly improper $D(1)$-triple from $\mathbb{Q}(i)[X]$ for which (\ref{jdbe_osnovne}) holds. Then, \begin{equation}\label{lm-dodano}\mp 2rst\equiv a+b-d_-\ ({\rm mod}\ {c}).\end{equation}
\end{lem}
\bigskip

In the proof of Theorem \ref{glavni-tm}, we follow the approach used in \cite[Theorem 1]{glavni} for $\mathbb{Z}[X]$ and in \cite[Theorem 1.4]{fij} for $\mathbb{R}[X]$, but the specificity of $\mathbb{Q}(i)[X]$ allows some more possibilities that must be examined. Although all possibilities for the initial terms from Lemma~\ref{fund_cases}
already occur over \(\mathbb R[X]\), many details require a different
approach over \(\mathbb Q(i)[X]\). 
\medskip

\begin{proof}[Proof of Theorem~\ref{glavni-tm}] In $\mathbb{Q}(i)[X]$, by Lemma \ref{gap0}, if $v_m=w_n$ and $d'=\frac{v_m^2-1}{c}$, then $m\geq 3$ and $n\geq 3$. The cases where $\{0,1,2\}\cap \{m,n\}\neq \emptyset$ are described in \cite[Proposition 1]{dij} for $\mathbb{C}[X]$. Thus, in the equality \(v_m=w_n\), we may assume that \(m\geq3\) and \(n\geq3\).\\

\textbf{Case 1.a)} $v_{2m}=w_{2n}$, $z_{0}=z_{1}=\pm 1$.

By (\ref{d0}),  (\ref{d1}) and Remark \ref{predznaci}, we take $x_0=1$ and $y_1=1$. By Lemma \ref{Lm8},
\begin{equation}\label{kon1}\pm am^2+ sm\equiv \pm bn^2+ tn\ ({\rm mod}\ {c}).\end{equation} We deal with two cases: $\beta < \gamma$ and $\beta = \gamma$. 

Let $\beta<\gamma$. Since $v_1=c\pm s$ and $w_1=c\pm t$, by (\ref{degvm}) and (\ref{degwn}), from ${\rm deg}(v_{2m})= {\rm deg}(w_{2n})$ similarly to \cite{dif2} we obtain a contradiction. 

By Lemma \ref{lemma-d_-}, we know that when $d_-\neq 0$ then if $\beta=\gamma$ we have $\alpha=0$. Moreover, in Remarks \ref{remark3} - \ref{remark5}, we precisely described such cases. Also, in Remark \ref{Rem_dminus} we described the possibility that $d_-=0$, which can occur for $\beta\leq \gamma$. Since the case \(\beta<\gamma\) has been excluded, it remains to consider \(d_-=0\) only when \(\beta=\gamma\).

Let $d_-=a$. For $a = \frac{1}{2}\left(f-\frac{1}{f}\right)$, where $f\in\mathbb{Q}(i)\backslash\{0\}$, analogously to \cite{fij} the congruence
$\pm am^2+sm\equiv \mp 2pn^2 - n \pmod{c}$, with $p\in\mathbb{Q}(i)$, 
is obtained from (\ref{kon1}). Since $\alpha < \gamma$, the resulting congruence is in fact an equality. However, the polynomial on the left-hand side is non-constant, whereas the polynomial on the right-hand side is constant. This yields a contradiction.

If $d_-=0$, according to Remark \ref{Rem_dminus}, if $c = a + b + 2r$, then $s = \pm (a+r)$ and $t = \pm (b+r)$, and when $c = a + b -2r$, then $s = \pm (a-r)$ and $t=\pm (b-r)$. Assume that  $\alpha<\beta$. From (\ref{kon1}), we get 
$$a(\pm m^2 \pm n^2 \pm m \pm n)\equiv \pm r(\mp 2n^2 \mp n \mp m) \pmod{c}.$$  
Since $a$ and $r$ have different degrees, both less than $\gamma$, both sides of the congruence must be equal to zero. This cannot happen for $2m\geq 3$ and $2n\geq 3$.

For $\alpha=\beta=\gamma$, analogously to \cite{dif2}, we get $d=0$ or $d=d_+$.\\  

\textbf{Case 1.b)} $v_{2m}=w_{2n}$, $z_{0}=z_{1}=\pm s$ and $\alpha=0$.

Using (\ref{d0}) and (\ref{d1}), we get $d_0=d_1=a$. Hence, $a=\frac{1}{2}\left(f - \frac{1}{f}\right)$ for $f \in \mathbb{Q}(i)\backslash\{0\}$. In this case $x_0^2=a^2+1$ and $y_1=\ r$. From Lemma \ref{Lm8}, we have 
\begin{equation}\label{kon2}\pm 2asm^2+2sx_0m\equiv \pm 2bsn^2+ 2trn\ ({\rm mod}\ {c}).\end{equation}
By multiplying the congruence (\ref{kon2}) by $s$ and using (\ref{jdbe_osnovne}) and (\ref{lm-dodano}), we get 
\begin{equation}\label{kon3}\pm 2am^2+ 2x_0m\equiv \pm 2bn^2+\varepsilon (a+b-d_-)n\ ({\rm mod}\ {c}),\end{equation}
where $\varepsilon\in\{-1,1\}$ is an independent sign.
\\
Let $\beta<\gamma$. From (\ref{kon3}), multiplying by $\pm 1$ (letting $\mu = \pm 1$ represent the sign of $z_0$) and using Lemma \ref{lemma-d_-}, 
\begin{equation}\label{kon4} 2am^2+2\mu x_0m =  2bn^2+ \mu \varepsilon (a+b-d_-)n.\end{equation}
If ${\rm deg}(d_-)<\beta$, by comparing the leading coefficients on both sides of this equation, we get $2n^2+\mu \varepsilon n=0$, which is not possible. Hence, $\rm{deg}(\textit{d}_{-})=\beta$ and by Lemma \ref{lemma-d_-}, we get $\gamma=2\beta$.  
According to Lemma \ref{lemma-d_-} we have $d_- = c+sv+tu$, so (\ref{kon3}) becomes
\begin{equation}
\pm 2am^2+2 x_0m \equiv \pm 2bn^2+ \varepsilon(a+b-sv-tu)n \pmod{c},
\end{equation}
from which, by considering degrees of polynomials, we conclude  
\begin{equation}
\pm 2am^2+2 x_0m \mp 2bn^2 - \varepsilon(a+b-sv-tu)n = k\cdot(d_--sv-tu), 
\end{equation} for some \(k\in\mathbb Q(i)\).

By analyzing the leading coefficients on both sides of the equation and using the fact that $\deg(sv+tu)=\deg(d_- - c)=\gamma$, we get $k = -\varepsilon n$. In
particular, \(k\in\mathbb Z\). So,
\begin{equation}\label{d_-1b}d_- = a + b \mp \frac{2am^2}{\varepsilon n} - \frac{2 x_0m}{\varepsilon n} \pm \frac{2bn}{\varepsilon}.\end{equation}
\\
One option from Lemma \ref{Lm-ubaceno_novo} is $d_- = a + b \pm 2r$. By (\ref{d_-1b}), it is easy to conclude that $\pm 2r \mp \frac{2bn}{\varepsilon}$ must be a constant polynomial then, which is not possible. Another option from Lemma \ref{Lm-ubaceno_novo} is $\{a, b, d_-\} = \left\{\frac{1}{2}\left(f-\frac{1}{f}\right), b, \frac{b}{f^2} - \frac{2}{f}\right\}, \ f\in\mathbb{Q}(i)\backslash\{0\}.$ By plugging $d_-$ into (\ref{d_-1b}), it now becomes
\begin{equation}
b\left(\frac{1}{f^2}-1 \mp \frac{2n}{\varepsilon}\right)
= a \mp \frac{2 a m^2}{\varepsilon n} - \frac{2 x_0 m}{\varepsilon n} + \frac{2}{f}.\label{usporedba}
\end{equation} 
We conclude that it has to be \begin{equation}\label{f}\frac{1}{f^2}=1 \pm 2n.\end{equation} 
Since $a= \frac{f}{2}(1-\frac{1}{f^2})$, using (\ref{f}), we get \begin{equation}\label{a}a= \mp nf.\end{equation}
From $x_0^2=a^2+1$ and Remark \ref{predznaci}, we further conclude $x_0=\frac{1}{2}(f+\frac{1}{f})$ and then, by (\ref{f}), \begin{equation}\label{x_0}x_0=f(1 \pm n).\end{equation}
In (\ref{usporedba}), the constant part is $a \mp \frac{2 a m^2}{\varepsilon n} - \frac{2 x_0 m}{\varepsilon n} + \frac{2}{f}=0$. By inserting (\ref{a}) and (\ref{x_0}), using \eqref{f}, we obtain \begin{equation}\label{zbroj}
    \mp n + 2m^2-\frac{2\varepsilon m}{n}(1\pm n)+ 2(1\pm 2n)=0.
\end{equation} Since $z_{0}=z_{1}=\pm s$, from (\ref{rekurzija_vm}) and (\ref{rekurzija_wn}), $v_1=c(x_0\pm a)\pm1$ and $w_1= cr\pm st$. Hence, $\deg(v_1) = \gamma$. By \eqref{eq3} $\deg (c^2r^2-s^2t^2)=2\gamma$, so we have $\deg(w_1) = \frac{3\beta}{2}$ or $\deg(w_1) = \frac{5\beta}{2}$, depending on the sign $\pm$. Now we compare the degrees of $v_{2m}$ and $w_{2n}$. If $2m = 3n-1$, then in \eqref{zbroj} we have $n \mid 1$, i.e., $n=1$. Then, $m=1$. However, the option \(m=n=1\) is not possible in \eqref{zbroj} for any
combination of signs. The second case, when $2m = 3n$, is also not possible in (\ref{zbroj}).\\
Let $\beta=\gamma$. By Lemma \ref{lemma-d_-}, $d_-=0$ and $c = a+b\pm 2r$ or $d_-=a$, where $a\in\mathbb{Q}(i)\backslash\{0\}$. For $d_-=0$, from (\ref{kon3}), we obtain $$\pm 2am^2 + 2x_0m \mp 2bn^2- \varepsilon(a+b)n=k(a+b\pm 2r), \ k\in\mathbb{Q}(i).$$ By comparing the leading coefficients on both sides of this equation, we get $\mp 2n^2-\varepsilon n-k=0$ and then $k=0$. Consequently, $n(\mp 2n- \varepsilon)=0$, which is not possible.
\\
Let $d_-=a$. By (\ref{kon3}) and Remark \ref{remark5}, we obtain
$$\pm 2am^2+2x_0m \mp 2bn^2-\varepsilon bn = k\left(\frac{b}{f^2} - \frac{2}{f}\right),$$
where $k,f\in \mathbb{Q}(i)$ and $f\neq 0$. Comparing the leading coefficients on both sides of this equation leads to  
\begin{equation}\label{papir1}\mp 2n^2- \varepsilon n-\frac{k}{f^2} = 0, \ \ \pm 2am^2 + 2x_0m + \frac{2k}{f}=0.\end{equation} 
From the first equation we get $k = f^2(\mp 2n^2 - \varepsilon n)$. Substituting this into the second equation and dividing by $\pm 2$ yields:
\begin{equation}\label{f-mm}
2n^2f \pm \varepsilon nf = am^2 \pm x_0m.
\end{equation}
If $\deg(v_1) = \gamma$ and $\deg(w_1) = \frac{3\gamma}{2}$, from $\deg(v_{2m}) = \deg(w_{2n})$ we obtain $m = 2n$. Substituting this relation into \eqref{f-mm} and dividing by $n$ yields equations for $f^2$ dependent only on $n$. After algebraic simplification, we obtain $n\not\in\mathbb{N}$ or $f^2 = \frac{\pm 2n-1}{2}$, which is not a square of $f\in\mathbb{Q}(i)\backslash\{0\}$ for $n\in\mathbb{N}$.
\\
If $\deg(v_1) = \gamma$ and $\deg(w_1) = \frac{\gamma}{2}$, from $\deg(v_{2m}) = \deg(w_{2n})$, we obtain $m = 2n-1$. After introducing $a=\frac{1}{2}(f-\frac{1}{f})$ and $x_0=\frac{1}{2}(f+\frac{1}{f})$ into (\ref{f-mm}) and simplifying, we get expressions for $f^2$ such as $\frac{(2n-1)(1-n)}{2n}$, $\frac{n(1-2n)}{4n-1}$ or $-n$.  If $f^2=\frac{(2n-1)(1-n)}{2n}$, then $f =-v^2$, where $v\in \mathbb{Q}$. Since factors $2n-1$, $n-1$, and $2n$ are coprime (depending on the parity of $n$), this fraction must be the ratio of two squares. 
For instance, if \(n=2k+1\), then this condition implies that both \(k\) and \(4k+1\) must be perfect squares. Setting $k = p^2$ and $4k+1 = q^2$ results in $(q-2p)(q+2p) = 1$. But, $p=0$ yields $n=1$, which contradicts $n \ge 3$. In the second case, setting $-v^2 = f^2$ gives $v^2 = \frac{n(2n-1)}{4n-1}$, where $v\in \mathbb{Q}$. Since $\gcd(n(2n-1), 4n-1)=1$, $4n-1$ has to be a perfect square. However, no perfect square can be of the form $4k-1$. Therefore, no rational $v$ can exist for any $n \in \mathbb{N}$. Let $f^2 = -n.$ Then
$\frac{1}{f} = -\frac{f}{n}$ so $a = \frac{f(n + 1)}{2n}$. Further, $c = \frac{s^2 - 1}{a}= \frac{2n(s^2 - 1)}{f(n + 1)}$. From $c = \frac{b - 2f}{f^2} = \frac{b - 2f}{-n}$, we obtain
$b = -cn + 2f= -\frac{2n^2(s^2 - 1)}{f(n +1)} + 2f.$ We compute 
$ab + 1 = -n s^2 - 1$, so $(r-fs)(r+fs)=-1,$ which is not possible for non-constant polynomials $r$ and $s$.\\

\textbf{Case 1.c)} $v_{2m}=w_{2n}$, $z_{0}=z_{1}=\pm (cr \pm st)$ and $\alpha>0$, $\alpha+\beta\leq\gamma\leq 2\alpha+\beta$.

In this case, we have $z_0 = z_1 = \pm w$ and consequently $x_0 = u = at \pm rs$ and $y_1 = v = bs \pm rt$. If we assume that $\beta = \gamma$, based on the inequalities between the degrees in this case, we obtain $\alpha = 0$, which contradicts the condition $\alpha > 0$.\\
Let $\alpha = \beta < \gamma$. We consider the relation $\deg(v_{2m}) = \deg(w_{2n})$ and include all four possibilities between degrees. Similarly to \cite{glavni}, we conclude that this situation is not possible. 

Let $\alpha < \beta < \gamma$. The equation \begin{equation}\label{koeficijenti}2[am(\pm m+1)-bn(\pm n+1)]=\pm (a+b-d_-)(n-m).\end{equation}  is obtained, similarly to \cite{dif2}, from Lemma \ref{Lm8} by taking into account the expressions of $z_0, z_1, x_0, y_1$ that occur in this case and by using the congruences $s^2 \equiv t^2 \equiv 1 \pmod{c}$ and the congruence (\ref{lm-dodano}). Since $\deg(d_-) \leq \alpha$, we have two possibilities to deal with. First let $\deg(d_-) < \alpha$. Considering the leading coefficients in (\ref{koeficijenti}) we get $-2n(\pm n +1) = \pm(n-m)$, $2m(\pm m +1) = \pm (n-m)$, and $\pm (n-m) =0$. From these equations, we get $(m, n) = \{(0, 0), (1, 1)\}$ which is not considered since we need $2m, 2n \geq 3$. The comparison of coefficients for the case when $\deg(d_-) = \alpha$ leads to the conclusion that $(m, n) = (1, 1)$, as in \cite{glavni}, which we again exclude because $2m, 2n \geq 3$. Hence, no choice of degrees $\alpha, \beta, \gamma$ yields a solution other than $d = d_+$.

We now explain how Case 1.c) also covers Cases 2.c), 3.d) and 4.c). After a possible change
of signs and a shift of one of the indices, the intersections appearing in Cases
2.c), 3.d) and 4.c) become intersections of the same form as in Case 1.c). The degree conditions, however, are inherited from the original cases and have to be checked separately. In Case 2.c) we have $(z_0,z_1)=(\pm t,\pm cr\pm st),$
$\alpha=\beta,$ $\gamma=3\alpha.$ This setup has already been treated in Case 1.c). In Case 3.d) we have $(z_0,z_1)=(\pm cr\pm st,\pm s),$
$2\alpha+\beta\leq\gamma\leq\alpha+2\beta.$ Therefore, we have to consider the range $2\alpha+\beta<\gamma\leq\alpha+2\beta$. In that case,
$\deg(d_-)=\gamma-\alpha-\beta$ satisfies $\alpha<\deg(d_-)\leq\beta.$ Hence, the possibilities $\alpha<\deg(d_-)<\beta,$ $\deg(d_-)=\beta$ have to be dealt with. Finally, in Case 4.c) we have
$(z_0,z_1)=(\pm t,\pm s),$ $\gamma\geq\alpha+2\beta$. Thus, $\deg(d_-)=\gamma-\alpha-\beta\geq\beta.$ After shifting both sequences by Lemma \ref{preklapanja}, this again gives a
Case 1.c)-type intersection, now with $\deg(d_-)=\beta \ \text{or}\ 
        \deg(d_-)>\beta.$ Consequently, Case 1.c) has to be understood in the extended sense, with all
possible relative positions of \(\deg(d_-)\) with respect to \(\alpha\) and
\(\beta\), namely
\[
        \alpha<\deg(d_-)<\beta,\qquad
        \deg(d_-)=\beta,\qquad
        \deg(d_-)>\beta .
\] Similarly to \cite{glavni}, these possibilities imply only $m = n = 0$ or $m = n = 1$, which is not possible. \\

\textbf{Case 2.a)} $v_{2m+1}=w_{2n}$, $(z_0,z_1)=(\pm 1, \pm s)$ and $\gamma \geq 2\alpha + \beta$.

By (\ref{d0}), (\ref{d1}) and Remark \ref{predznaci}, we get $x_0=1$, $y_1=r$. By Lemma~\ref{Lm8}, since $(s,c)=1$, we conclude that $z_0=\pm 1$ and $z_1=\pm s$ have the same signs and then, by (\ref{lm-dodano}),
\begin{equation}\label{kon5}\pm 2am(m+1)+s(2m+1)\equiv \pm 2bn^2+\varepsilon(an+bn-d_-n)\ ({\rm mod}\ {c}),\end{equation}where $\varepsilon\in\{-1,1\}$ is an independent sign. We will split this part of the proof into two parts, $\beta<\gamma$ and $\beta=\gamma$. 

Let $\beta<\gamma$. Since we assume that $2n\geq 3$, by (\ref{kon5}), after multiplying by $\varepsilon$ we obtain \begin{equation}\label{jed5}d_-=\pm 2bn+a+b\mp 2a\frac{m(m+1)}{n}-\varepsilon s\frac{2m+1}{n}.\end{equation} Lemma~\ref{lemma-d_-} implies $d_-\neq 0$ and, by (\ref{jed5}), $\gamma\leq 3\alpha+2\beta.$ By Lemma \ref{lemma-d_-}, in this case $\deg(d_{-})\geq \alpha$. Since $\alpha<\gamma$, according to (\ref{jed5}),  $$\gamma-\alpha-\beta=\deg(d_-) \leq \max\left\{\ \beta, \ \frac{\alpha+\gamma}{2}\right\}.$$  If $\beta > \frac{\alpha + \gamma}{2}$, then $\deg(d_-) = \beta$ i.e. $\beta = \frac{\gamma-\alpha}{2}.$ After introducing $\frac{\gamma-\alpha}{2}>\frac{\alpha+\gamma}{2}$, we get $\alpha<0$, which is not possible. If $\beta < \frac{\alpha+\gamma}{2}$, then $\deg(d_-) = \frac{\alpha+\gamma}{2}$. By Lemma \ref{lemma-d_-}, $\beta = \frac{\gamma-3\alpha}{2}$ is obtained. We have $\alpha >0$, since otherwise we obtain a contradiction with the assumptions on \(\beta\). Since $\beta>0$, we easily conclude $\gamma>3\alpha$. In this case, $\deg(v_1) = \gamma$ and $\deg(w_1) = \frac{5\gamma-\alpha}{4}$. Lemma \ref{lema_stupnjevi} implies \begin{equation}\label{alpha-gamma}
\frac{2m}{3n-1} = \frac{\gamma-\alpha}{\gamma+\alpha}<1.    
\end{equation}If $\alpha=\beta$, then $\gamma=5\alpha$ and \eqref{alpha-gamma} implies $3m=3n-1$, which is not possible. Hence, $\alpha<\beta<\gamma.$ The congruence obtained from Lemma \ref{Lm8}, after dividing by $c$ becomes 
\begin{equation}\label{kon5-new}\pm 2asm(m+1)+2m+1\equiv  2(\pm bsn^2+trn)\ ({\rm mod}\ {c}).\end{equation}Since $\beta+\frac{\alpha+\gamma}{2}=\gamma-\alpha<\gamma$, then \eqref{kon5-new} forces the equation \begin{equation}\label{in}2trn=s(\pm 2m(m+1)a\mp 2bn^2)-2m-1.\end{equation}By comparing the leading coefficients on both sides of the equation \eqref{in}, we get $n=\mp n^2$, which is not possible. 

Let $\beta=\frac{\alpha+\gamma}{2}$. Assume first that $\deg(d_-) = \beta$. In this case $\alpha =0$ and $\gamma = 2\beta$. From $\deg(v_{2m+1}) = \deg(w_{2n})$, we get $2m\beta+\deg(v_1) = (2n-1)\frac{3\beta}{2}+\deg(w_1).$ Since $v_1 = \pm s + c$ and $w_1 = \pm st + cr$, we conclude $\deg(v_1) = \gamma = 2\beta$ and $\deg(w_1) \in \left\{\frac{5\beta}{2}, \frac{3\beta}{2}\right\}$. So, $2\beta m + 2\beta = (2n-1)\frac{3\beta}{2} + \frac{5\beta}{2}$ or $2\beta m + 2\beta = (2n-1)\frac{3\beta}{2} + \frac{3\beta}{2}$, leading to $2m - 3n = -1$ or $2m - 3n = -2$, respectively. In this case, we are dealing with two possibilities described in Lemma \ref{Lm-ubaceno_novo}, namely $d_- = a + b \pm 2r$ or $d_- = \frac{1}{f^2}b- \frac{2}{f}$. If $d_-=a + b \pm 2r$, by Lemma \ref{lemma-3.7-param}, $s=\pm(2a(b+r)+1)$,  so $\deg(s) = \beta$. From (\ref{jed5}), we obtain
\begin{equation}
\pm 2r = \pm 2bn \mp 2a\frac{m(m+1)}{n} +\varepsilon(2a(b+r)+1)\frac{2m+1}{n}.
\end{equation} By comparing the leading coefficients we conclude that $\pm n+\varepsilon a \frac{2m+1}{n}=0$ and then, by comparing coefficients of the remaining polynomials of degree $\frac{\beta}{2}$, we get $\pm 1 = \varepsilon a\frac{2m+1}{n}$. Therefore, $n=1$, which contradicts the fact that $2n\geq 3$. \\
Now, let $d_- = \frac{1}{f^2}b- \frac{2}{f}$. We introduce $d_-$ into (\ref{jed5}) and, by Lemma \ref{lemma-3.7-param}, obtain
\[
\left(\frac{1}{f^2} - 1 \mp 2 n\right) b = a \mp 2 a \frac{m(m+1)}{n} +\varepsilon \left( b\frac{1-f^2}{f^2} + \frac{f^2-3}{2f} \right) \left(\frac{2m+1}{n}\right)+ \frac{2}{f}.
\]
By comparing the corresponding coefficients and introducing $a=\frac{f^2-1}{2f},$
we obtain the system
\[
\begin{cases}
n(1-f^2\mp 2nf^2)
=
\varepsilon(1-f^2)(2m+1),\\[6pt]
(f^2-1)\bigl(n\mp 2m(m+1)\bigr)
+
\varepsilon(f^2-3)(2m+1)
+
4n
=0.
\end{cases}
\]
We now consider the cases $2m-3n=-1$ and $2m-3n=-2.$
In each case we express $m$ in terms of $n$, substitute this into the system obtained by comparing the corresponding coefficients, and then distinguish the possibilities $\varepsilon=\pm1$ and the two choices of signs (upper and lower signs).
We illustrate the computation for
$2m-3n=-1,\ \varepsilon=1.$
The system of the equations becomes
\[
\begin{cases}
n(1-f^2\mp 2nf^2)=3n(1-f^2),\\
(f^2-1)\left(n\mp \dfrac{9n^2-1}{2}\right)
+3n(f^2-3)+4n=0.
\end{cases}
\]
Taking the upper and lower signs separately, we obtain
$-3n(3n-1)(n-1)=0$ or $-3n(3n+1)(n+1)=0$, respectively.
The corresponding roots do not yield an admissible positive integer value of $n$: either $n=0$ or $n\notin\mathbb N$.
The remaining cases are analogous. For $2m-3n=-1$ and $\varepsilon=-1$, one obtains
$3n(3n^2+8n+13)=0$ or $-3n(3n^2-8n+13)=0.$
For $2m-3n=-2$ and $\varepsilon=1$, one obtains
$9n^4-18n^3+10n^2-6n+2=0$ or $9n^4+6n^3-2n^2-6n+2=0,$
while for $2m-3n=-2$ and $\varepsilon=-1$, one obtains
$9n^4+18n^3+34n^2-18n+2=0$ or $9n^4-30n^3+46n^2-18n+2=0.$
By the Rational Root Theorem, together with the preceding exclusions in the factored cases, none of these equations yields an admissible positive integer value of $n$. 

Assume now that $\beta=\frac{\alpha+\gamma}{2}$ and $\deg(d_-) = \beta - 2\alpha < \beta$, which implies $\alpha > 0$. Since $\gamma=2\beta-\alpha \geq 2\alpha+\beta$, we have $\beta \geq 3\alpha$, easily yielding $\alpha < \beta < \gamma$.\\
We introduce $z_0=\delta$ and $z_1=\delta s$, where $\delta\in\{-1,1\}$. In this case, $x_0 = 1$ and $y_1 = r$, so $v_1=\delta s+c$ and $w_1=\delta st+cr$. 
Since $\deg(s) = \frac{\alpha+\gamma}{2} = \beta$ and $\deg(t) = \frac{\beta+\gamma}{2} = \frac{3\beta-\alpha}{2}$, the initial values yield $\deg(v_1) = \gamma = 2\beta-\alpha$. 
We first determine which sign $\delta$ in $w_1$ corresponds to each possible value of $\deg (w_1)$. From (\ref{dminus}), we obtain
$$2r(st-\varepsilon cr) = 2rst-2\varepsilon cr^2 = \varepsilon(a+b+c+2abc-d_-) - 2\varepsilon c(ab+1) = \varepsilon(a+b-c-d_-).$$
The right-hand side has degree $\gamma$, and then $\deg(st-\varepsilon cr)=\frac{3\beta-3\alpha}{2}$. Polynomial $st-\varepsilon cr$ can't have maximal possible degree, which is $\frac{5\beta-\alpha}{2}>\gamma$. Also, $\pm w_1=\delta st+cr$ precisely when $\delta\varepsilon=-1$. We set $\eta=\delta\varepsilon\in\{-1,1\}$. Thus, $\deg(w_1)=\frac{5\beta-\alpha}{2}$ corresponds to $\eta=1$, and $\deg(w_1)=\frac{3\beta-3\alpha}{2}$ corresponds to $\eta=-1$.\\
Let $M=m(m+1)$ and $K=2m+1$. From \eqref{kon5} using \eqref{lm-dodano}, multiplying by $s$, and applying $s^2=ac+1$, we obtain 
\begin{align} 
2\delta aM+sK &\equiv 2\delta bn^2+\varepsilon n(a+b-d_-) \pmod{c}, \label{2a31} \\ 
2\delta asM+K &\equiv 2(\delta bsn^2+trn) \pmod{c}. \label{2a32} 
\end{align} 
Since all terms in \eqref{2a31} have degree less than $\gamma$, congruence \eqref{2a31} is an equality. Define $D=2\delta asM+K-2\delta bsn^2-2trn$. Then \eqref{2a32} implies $c\mid D$. Multiplying $D$ by $s$, and applying $s^2=ac+1$, \eqref{dminus}, and the equality obtained from \eqref{2a31}, yield $sD = c\Bigl(\delta(2a^2M-2abn^2)-\varepsilon n(1+2ab)\Bigr).$
Since $c\mid D$, we get $s\mid (\delta(2a^2M-2abn^2)-\varepsilon n(1+2ab))$. Multiplying by the unit $\delta$, we obtain $s\mid Q_\eta$, where 
\[ Q_\eta = 2a^2M-2abn^2-\eta n(1+2ab). \]
Furthermore, since $\pm$ signs in (\ref{jed5}) correspond precisely to our choice of $\eta$, by substituting $M = m(m+1)$ and multiplying by $n$, we obtain
\begin{equation*}
\varepsilon s (2m+1) = \eta 2bn^2 + bn + an - n d_- - \eta 2aM.
\end{equation*}
We conclude that $s \mid L_\eta$, where $L_\eta = \eta 2bn^2 + bn + an - n d_- - \eta 2aM.$

Therefore, $s \mid Q_\eta$ and $s \mid L_\eta$, and we now consider the two possibilities for $\eta$.\\
Firstly, suppose $\deg (w_1)=\frac{5\beta-\alpha}{2}$, which corresponds to $\eta=1$.
In this case, from $\deg(v_{2m+1}) = \deg(w_{2n})$, we obtain
$(2m - 3n + 1)\beta = (1 - n)\alpha.$
For $\eta=1$, we have $L_1=n(2n+1)b+an-nd_- -2aM$ and $Q_1=2a^2M-2abn(n+1)-n$. Since $s\mid L_1$ and $s\mid Q_1$, we conclude $s\mid \bigl((2n+1)Q_1+2a(n+1)L_1\bigr)$. A direct simplification gives
\[
\begin{aligned}
    (2n+1)Q_1+2a(n+1)L_1 &= 2a^2\bigl(n(n+1)-M\bigr) -2an(n+1)d_- - n(2n+1).
\end{aligned}
\]
Denote this right-hand side by $R_1$. Since $2\alpha<\beta$ and $\deg(ad_-)=\beta-\alpha<\beta$, $\deg(R_1)<\deg(s)=\beta$. Since $s\mid R_1$, we conclude $R_1=0$. Therefore, 
$$2a\Bigl(a\bigl(n(n+1)-m(m+1)\bigr)-n(n+1)d_-\Bigr)=n(2n+1).$$
The left-hand side is divisible by the non-constant polynomial $a$, while the right-hand side is a positive integer. Hence, this possibility cannot occur.\\
It remains to consider the case $\deg (w_1)=\frac{3\beta-3\alpha}{2}$. Then $\eta=-1$, and the degree comparison gives $(2m-3n+2)\beta=-n\alpha$. In this case, using that $s\mid L_{-1}$ and $s\mid Q_{-1}$, we form a suitable linear combination in which the terms containing $b$ cancel. This gives a polynomial $R_{-1}$ such that $s\mid R_{-1}$, while all terms of $R_{-1}$ have degree strictly lower than $\beta=\deg (s)$. This, again, implies that the non-constant polynomial $a$ divides a nonzero constant, so we get a contradiction. Consequently, this subcase is not possible.

For $\beta = \gamma$, $\alpha = 0$. According to Lemma \ref{lemma-d_-}, $d_- = 0$ or $d_- = a$, where $a \in \mathbb{Q}(i)\backslash\{0\}$. Let $d_- = 0$. By Remark \ref{Rem_dminus}, the congruence (\ref{kon5}) becomes
$$ \pm 2am(m+1)+(\pm a\pm r)(2m+1) \mp 2bn^2 - \varepsilon(an+bn) = k(a+b\pm 2r), \  k\in\mathbb Q(i).$$
Comparing the corresponding coefficients on both sides of the equation gives
$$ \pm 2m(m+1) \pm (2m+1) - \varepsilon n = k, \quad \mp 2n^2 - \varepsilon n = k, \quad \text{and} \quad \pm(2m+1) = \pm 2k. $$ The first two equalities show, in particular, that \(k\in\mathbb Z\). From equation $\pm 2m\pm 1 = \pm 2k$, we immediately reach a contradiction.  Therefore, this case cannot occur.

Similarly, for $d_- = a$, we obtain $2m+1 = 0$ which is not possible.\\ 

\textbf{Case 2.b)} $v_{2m+1} = w_{2n}, \ (z_0, z_1) = (\pm s, \pm 1), \ \alpha = 0, \ \beta \leq \gamma.$

By Remark \ref{constant}, we have $a = \frac{1}{2}\left(f-\frac{1}{f}\right)$ and $x_0 = \frac{1}{2}\left(f+\frac{1}{f}\right)$, with $f\in\mathbb{Q}(i)\setminus\{0\}$. From Lemma \ref{Lm8}, we obtain the congruence
\begin{equation}\label{2b1}
\pm a \pm 2am(m+1) + x_0(2m+1) \equiv \pm 2bn^2 + 2tn \pmod{c},
\end{equation}
and we also conclude that $z_0$ and $z_1$ have the same sign $\pm$, since otherwise we obtain $c\mid 2$, a contradiction. Since \(\alpha=0\), the left-hand side of \eqref{2b1} is a constant. 
We have $2n\geq 3$. If \(\beta<\gamma\), then the term \(2tn\) has degree
$\deg (t)=\frac{\beta+\gamma}{2},$ whereas \(\deg (b)=\beta<\frac{\beta+\gamma}{2}\). Hence, the right-hand side has degree \((\beta+\gamma)/2<\gamma\). Therefore, the congruence \eqref{2b1} is an equality, which is impossible since its
left-hand side is constant.

Let \(\beta=\gamma\). Since \(\alpha=0\), by
Lemma~\ref{lemma-d_-}, either \(d_-=0\) or \(d_-=a\).
\begin{enumerate}
\item Let \(d_-=0\), and \(\eta,\varepsilon\in\{\pm1\}\). According to Remark~\ref{Rem_dminus}, $c=a+b+2\eta r,\ t=\varepsilon(b+\eta r).$
Hence, congruence \eqref{2b1} becomes
$$\pm a \pm 2am(m+1)+x_0(2m+1)
\mp 2bn^2 -2\varepsilon(b+\eta r)n=k(a+b+2\eta r),$$
for some \(k\in\mathbb Q(i)\). Comparing the coefficients of \(b\) and \(r\), respectively, gives $\mp 2n^2-2\varepsilon n=k,  \ -2\varepsilon\eta n=2\eta k.$ 
Hence \(k=-\varepsilon n\), and consequently $n(\mp2n-\varepsilon)=0.$
Since \(2n\geq3\), that is not possible. 
\item Let $ d_-=a=\frac12\left(f-\frac1f\right), \ f\in\mathbb Q(i)\setminus\{0\},$ 
and, in the notation of Remark~\ref{remark5}, $c=\frac{b}{f^2}-\frac2f, \ t=1-\frac bf.$ Since \(z_0=\pm s\), we have $ v_1=c(x_0\pm a)\pm1.$
Moreover, $x_0=\frac12\left(f+\frac1f\right),$ and therefore $x_0+a=f,\ x_0-a=\frac1f.$ Thus \(x_0\pm a\neq0\), and $\deg(v_1)=\gamma.$
Since \(z_1=\pm1\) and \(y_1=1\), we get $w_1=\pm t+c=\left(\mp\frac{1}{f}+\frac1{f^2}\right)b
\pm 1-\frac2f.$ The leading term could cancel only if $\mp\frac{1}{f}+\frac1{f^2}=0,$
that is, only if \(f=\pm 1\). But then $a=\frac12\left(f-\frac1f\right)=0,$
which is impossible. Hence, $\deg(w_1)=\gamma.$ By Lemma~\ref{lema_stupnjevi}, and since
\(\alpha=0\), we obtain  $\deg(v_{2m+1})=(m+1)\gamma.$
On the other hand, since \(\beta=\gamma\), we obtain $\deg(w_{2n})=2n\gamma.$
The equality \(v_{2m+1}=w_{2n}\) implies \begin{equation}\label{2b}m+1=2n.\end{equation} 
 We have $b=f^2c+2f$ and $t=-fc-1$. So, $b\equiv 2f \pmod c,
\ t\equiv -1 \pmod c,$ and \eqref{2b1} is a congruence between constants modulo
the non-constant polynomial \(c\). Hence, it is an equality. We get
\begin{equation}\label{eq:constant-equality}
\pm a\bigl(1+2m(m+1)\bigr)+x_0(2m+1)
=
\pm 4fn^2-2n.
\end{equation} Substituting \eqref{2b} and expressions for $a$ and $x_0$, we distinguish two cases. Since \(2n\geq3\), i.e., \(n\geq2\), the equation \eqref{eq:constant-equality} with the upper signs yields \begin{equation}\label{f-2b}
f=\frac{(2n-1)^2}{2n}.
\end{equation} For the lower signs in \eqref{eq:constant-equality}, equation \[
(4n-1)f^2+2nf+4n^2=0
\] arises. The discriminant of this quadratic equation is $\Delta=-4n^2(16n-5).$ Since \(f\in\mathbb Q(i)\), we conclude that \(16n-5\) is a perfect square. But, $16n-5\equiv 11 \pmod {16},$
whereas a square modulo \(16\) can only be \(0,1,4,\) or \(9\). Hence, this is impossible. It remains to rule out \eqref{f-2b}.

We compare leading coefficients in the equality $v_{2m+1}=w_{2n}$. By comparing the leading coefficients in $s^2=ac+1$, we get $S^2=aC.$ In the remaining case we have the upper signs, so $v_1=c(x_0+a)+1=fc+1$,
and, since $t=-fc-1$, $w_1=t+c=(1-f)c-1.$ By \eqref{rekurzija_vm} and \eqref{rekurzija_wn}, the leading coefficients of $v_{2m+1}$ and $w_{2n}$ are  
\begin{eqnarray*}(2S)^{2m}fC\ \ \textup{and}\ \ (2(-fC))^{2n-1}(1-f)C,\end{eqnarray*} respectively. Since $m=2n-1$, the equality 
$v_{2m+1}=w_{2n}$ gives $2^{4n-2}a^{2n-1}fC^{2n} =
    2^{2n-1}f^{2n-1}(f-1)C^{2n}.$
After cancellation, we have $2^{2n-1}a^{2n-1}=f^{2n-2}(f-1).$ Substituting $a=\frac12\left(f-\frac1f\right)=\frac{f^2-1}{2f},$
we get $(f^2-1)^{2n-1}=f^{4n-3}(f-1).$ Since $f\neq1$, this becomes
\begin{equation}\label{eq:star}
    (f-1)^{2n-2}(f+1)^{2n-1}=f^{4n-3}.
\end{equation} Now put $p=2n-1.$
Then $p\geq3$ is odd, and \eqref{f-2b} gives $f=\frac{p^2}{p+1}.$ Substituting this into \eqref{eq:star} and clearing denominators, gives
\[
    (p^2-p-1)^{p-1}(p^2+p+1)^p=p^{4p-2}.
\] Since $p-1$ is even, reducing this congruence modulo $p$, we obtain $1\equiv 0 \pmod p$. This contradiction rules out the remaining case.
\end{enumerate}

\vskip 1em 

\textbf{Case 2.c)} $v_{2m+1} = w_{2n}, \ (z_0, z_1) = (\pm t, \pm cr \pm st)$, $\alpha = \beta, \ \gamma = 3\alpha$.\\
By Lemma \ref{preklapanja}, this case is reduced to the extended form of Case 1.c), as explained at the very end of Case 1.c).\\

\textbf{Case 3.a)} $v_{2m} = w_{2n+1}, \ (z_0, z_1) = (\pm t, \pm 1), \gamma \geq \alpha + 2\beta. $

In this case $x_0 = r$ and $y_1 = 1$. From the congruence obtained by Lemma \ref{Lm8}, we first conclude that the signs $\pm$ of $z_0$ and $z_1$ are equal. Otherwise $c\mid t$, which is not possible because of $bc+1=t^2$. Dividing that congruence by $c$ yields:
\begin{equation}\label{3a1}
\pm 2 atm^2 +2rsm \equiv \pm 2btn(n+1) + 2n+1 \pmod{c}.
\end{equation}
Multiplying by $t$ and applying $t^2 \equiv 1 \pmod{c}$ alongside (\ref{lm-dodano}) gives:
\begin{equation}\label{3a2}
    \pm 2am^2 +\varepsilon(a+b-d_-)m \equiv \pm 2bn(n+1)+t(2n+1) \pmod{c},
\end{equation}
where $\varepsilon\in\{-1,1\}$. Since $\beta<\gamma$, $\deg(d_-) = \gamma - \alpha - \beta < \gamma$, and $\deg(t) = \frac{\beta+\gamma}{2}<\gamma$, both sides of the congruence (\ref{3a2}) have degrees strictly less than $\gamma$. Thus, the congruence is an equality:
\begin{equation}\label{3a3}
\pm 2am^2 +\varepsilon(a+b-d_-)m = \pm 2bn(n+1)+t(2n+1).
\end{equation} Since $\deg (t)=\frac{\beta+\gamma}{2}>\beta$ and $2n+1\neq0$, the right-hand side of \eqref{3a3} has degree 
$\frac{\beta+\gamma}{2}$. Moreover, $m\neq0$ and
$\deg (d_-)=\gamma-\alpha-\beta\geq\beta$, since $\gamma\geq\alpha+2\beta$. If $\deg (d_-)=\beta$, then the left-hand side of \eqref{3a3} would
have degree $\leq\beta$, while the right-hand side has degree
$\frac{\beta+\gamma}{2}>\beta$, which is impossible. Hence,
$\deg (d_-)>\beta\geq\alpha.$
Therefore, $-\varepsilon m d_-$ is the unique term of maximal
degree on the left-hand side of \eqref{3a3}. It follows that
$\deg (d_-)=\deg (t).$ Consequently, $$\gamma=2\alpha+3\beta,\qquad
\deg (d_-)=\alpha+2\beta.$$

We now analyze the possible values for $\alpha$:
\begin{itemize}
\item If $\alpha \neq \beta$ and $\alpha \neq 0$: This implies $\alpha < \beta$, meaning $\deg(at) < \deg(bt) < \gamma$. The initial congruence (\ref{3a1}) becomes an equality. Comparing leading coefficients yields $n(n+1) = 0$, which contradicts our initial assumptions.
\item If $\alpha = \beta$, then $\gamma=5\alpha$ and
$\deg(d_-)=3\alpha.$ Furthermore,
$\deg(t)=\frac{\beta+\gamma}{2}=3\alpha<5\alpha=\gamma$. Therefore, since $w_1=\pm t+c$, $\deg(w_1)=5\alpha.$ By \eqref{degwn}, $\deg(w_{2n+1})
=2n\frac{\beta+\gamma}{2}+\deg(w_1)
=(6n+5)\alpha.$ On the other hand, $\deg(cr)=\deg(st)=6\alpha.$ For the noncancelling choice of $v_1=cr\pm st$, we therefore have
$\deg(v_1)=6\alpha,$ 
and hence, by \eqref{degvm}, $\deg(v_{2m})
=(2m-1)\frac{\alpha+\gamma}{2}+\deg(v_1)
=(6m+3)\alpha.$ The equality $v_{2m}=w_{2n+1}$ then implies $6(m-n)=2$,
which is impossible. For the cancelling choice of $v_1=cr\pm st$, we use the identity $(cr\pm st)^2=cd_-+1.$ Since
$\deg(cd_-)=\gamma+\deg(d_-)=8\alpha>0,$
we obtain $\deg(v_1)
=\deg(cr\pm st)
=\frac{\deg(cd_-)}{2}
=4\alpha.$
Thus $\deg(v_{2m})
=(2m-1)3\alpha+4\alpha
=(6m+1)\alpha.$
The equality $v_{2m}=w_{2n+1}$ now gives $6(m-n)=4$, which is again impossible. 

\item Suppose that $\alpha=0$. Then $\gamma=3\beta,\ \deg (d_-)=2\beta.$
Since $\deg (t)=2\beta<3\beta$, we have
$\deg (w_1)=3\beta$, and hence $\deg (w_{2n+1})=(4n+3)\beta$. If $v_1$ is the noncancelling combination of $cr$ and $st$, then
$\deg (v_1)=\frac{7\beta}{2}$, and therefore
$\deg (v_{2m})=(3m+2)\beta.$ Thus $3m=4n+1$. If $v_1$ is the cancelling combination, then
$v_1^2=cd_-+1,$
so $\deg (v_1)=\frac{5\beta}{2}$, and consequently $\deg (v_{2m})=(3m+1)\beta,$ which gives $3m=4n+2$. By applying Lemma~\ref{lemma-3.7-param} to the triple $\{a,b,d_-\}$, there is an 
$\eta\in\{\pm1\}$ such that $d_-=4r(r+\eta a)(b+\eta r).$
Since $r^2=ab+1$ and $a$ is a nonzero constant, $D_-=4AB^2$. Put
\[
   u=2r^2+2\eta ar-1,\qquad
   v=2br+\eta(2r^2-1).
\]
Then $u^2=ad_-+1$ and $v^2=bd_-+1$. Since $c$ is the large
regular extension of $\{a,b,d_-\}$, a square root of $bc+1$ is,
up to sign, the noncancelling combination of $t=bu\pm rv.$ It follows that $T
   =\pm 4AB^2
   =\pm D_-.$ 
In \eqref{3a3}, the unique terms of degree $2\beta$ are
$-\varepsilon m d_-$ on the left and $(2n+1)t$ on the right. Comparing their leading coefficients gives $-\varepsilon mD_-=(2n+1)T$ and leads to $m=2n+1.$
Combining this with either $3m=4n+1$ or $3m=4n+2$ gives,
respectively, $2n=-2$ or $2n=-1$, both impossible.
\end{itemize}

\vskip 0.25cm

\textbf{Case 3.b)} \(v_{2m}=w_{2n+1}\), \((z_0,z_1)=(\pm s,\pm1)\),
\(\alpha=0\), \(\beta=\gamma\).

From \(cd_0+1=z_0^2=s^2=ac+1\), we get
\(d_0=a\). Thus \(x_0^2=a^2+1\), and \(x_0\in\mathbb Q(i)\). Since
\(z_1=\pm1\), we take \(y_1=1\). Lemma~\ref{initial1} gives \(z_0\equiv tz_1\pmod c\), hence
\(\pm s\equiv\pm t\pmod c\). Since $\beta=\gamma$, by Lemma \ref{lemma-d_-}, $d_-=0$ or $d_-=a$. If \(d_-=a\), by Remark \ref{remark5},
\(c=\frac{b}{f^2}-\frac2f\) and \(t=1-\frac bf=-fc-1\), so
\(t\equiv-1\pmod c\). Therefore \(c\mid(s\pm1)\), which is impossible since
\(\deg(s\pm1)=\frac{\gamma}{2}<\gamma\). Hence, \(d_-=0\). 
Thus \(c=a+b+2\eta r\), \(s=\rho(a+\eta r)\), and
\(t=\kappa(b+\eta r)\), where \(\eta,\rho,\kappa\in\{\pm1\}\). Write
\(z_0=\lambda s\) and \(z_1=\mu\), with
\(\lambda,\mu\in\{\pm1\}\), and put \(U=a+\eta r\). Since
\(c=a+b+2\eta r\), we have \(b+\eta r=c-U\). Hence,
\(s=\rho U\) and \(t=\kappa(c-U)\).  Therefore,
\[
\begin{aligned}
z_0-tz_1
=\lambda s-\mu t
=(\lambda\rho+\mu\kappa)U-\mu\kappa c.
\end{aligned}
\]
Since \(\lambda,\mu,\rho,\kappa\in\{\pm1\}\), we have
\(\lambda\rho+\mu\kappa\in\{-2,0,2\}\). As \(c\mid(z_0-tz_1)\), it follows that
\(c\mid(\lambda\rho+\mu\kappa)U\).  If \(\lambda\rho+\mu\kappa\neq0\), then \(c\mid U\).
But $\deg U=\frac{\beta}{2}<\beta=\gamma,$
which is impossible. Therefore \(\lambda\rho+\mu\kappa=0\), and hence
\(z_0-tz_1=-\mu\kappa c=\pm c\). Applying Lemma~\ref{Lm8} to \(v_{2m}=w_{2n+1}\), and dividing by \(c\), we obtain 
\[
        \pm1+2az_0m^2+2sx_0m
        \equiv
        2btz_1n(n+1)+2n+1
        \pmod c.
\]
Multiplying by \(s\), and using \(s^2\equiv1\pmod c\) and
\(st\equiv\pm1\pmod c\), this becomes
\[
        (\omega-(2n+1))s+q-2\theta bn(n+1)\equiv0\pmod c,
\]
where \(\omega,\theta\in\{\pm1\}\) and $q =2\lambda am^2+2x_0m\in \mathbb{Q}(i)$. Since \(s=\rho(a+\eta r)\) and \(c=a+b+2\eta r\), there exists
\(k\in\mathbb Q(i)\) such that
\[
        (\omega-(2n+1))\rho(a+\eta r)+q-2\theta bn(n+1)
        =
        k(a+b+2\eta r).
\] Comparing first the terms of degree \(\beta\), and then the terms of degree
\(\frac{\beta}{2}\), gives $k=-2\theta n(n+1),\ 2k=\rho(\omega-(2n+1)).$
Hence, $-4\theta n(n+1)=\rho(\omega-(2n+1)).$
If \(\omega=1\), then \(2\theta(n+1)=\rho\), impossible. If \(\omega=-1\),
then \(2\theta n=\rho\), again impossible. Therefore, Case 3.b) gives no
admissible irregular extension.\\

\textbf{Case 3.c)} \(v_{2m}=w_{2n+1}\), \((z_0,z_1)=(\pm1,\pm1)\),
\(\alpha=0\), \(\beta=\gamma\).

From \(cd_0+1=z_0^2=1\) and \(cd_1+1=z_1^2=1\), we get \(d_0=d_1=0\).
Hence, we may take \(x_0=y_1=1\). Lemma~\ref{initial1} gives 
\(\pm1\equiv\pm t\pmod c\). Since $\deg t=\frac{\beta+\gamma}{2}=\gamma$, we have \(t=kc\pm1\), for some
\(k\in\mathbb Q(i)\). Squaring this identity and using \(t^2=bc+1\), we get
\(bc+1=k^2c^2\pm2kc+1\), and hence
\begin{equation}\label{3c1}
        b=k^2c\pm2k.
\end{equation} 

Since $\beta=\gamma$, by Lemma \ref{lemma-d_-}, $d_-=0$ or $d_-=a$. If \(d_-=0\), then \(c=a+b\pm2r\). Since \(\deg (r)=\frac{\beta}{2}<\beta\),
the polynomials \(b\) and \(c\) have the same leading coefficient. Comparing the leading
coefficients in \((\ref{3c1})\), we obtain \(k^2=1\). Thus \(k=\pm1\), and
consequently \(b=c\pm2\). Combining this with \(c=a+b\pm2r\), we get that
\(2r\) is constant, which is impossible since \(\deg (r)=\frac{\beta}{2}>0\).

If \(d_-=a\), then
\(c=\frac{b}{f^2}-\frac2f\) and \(t=1-\frac bf=-fc-1\). Since \(z_0=\pm1\),
we have \(v_1=\pm s+c\), and therefore \(\deg(v_1)=\gamma\). Also,
\(w_1=\pm t+c\). The leading terms of \(\pm t\) and \(c\) could cancel only
if \(f=\pm1\), but then \(a=\frac12(f-\frac1f)=0\), impossible. Hence
\(\deg(w_1)=\gamma\). By Lemma~\ref{lema_stupnjevi}, we get
\(\deg(v_{2m})=\frac{(2m+1)\gamma}{2}\) and
\(\deg(w_{2n+1})=(2n+1)\gamma\). The equality \(v_{2m}=w_{2n+1}\) would imply
\(2m+1=2(2n+1)\), which is impossible. Therefore Case 3.c) gives no admissible irregular extension.\\

Case \textbf{3.d)} $v_{2m} = w_{2n+1}, \ (z_0, z_1) = (\pm cr \pm st, \pm s), \ \alpha \geq 0, \ 2\alpha+\beta \leq \gamma \leq \alpha+2 \beta$

Lemma \ref{preklapanja} reduces this case to the extended form of Case 1.c), as explained at the end of Case 1.c).\\

\textbf{Case 4.a)} $v_{2m+1} = w_{2n+1}$, $(z_0, z_1) = (\pm 1, \pm cr\pm st)$, $\gamma \leq 2\alpha + \beta$.

By Lemma~\ref{preklapanja}, the equality $v_{2m+1}=w_{2n+1}$, $(z_0,z_1)=(\pm1,\pm cr\pm st)$ can be rewritten, up to admissible changes of  signs, in the form $v_{\pm1,2m+1}=w_{\pm s,2N},$
where \(N=n\) or \(N=n+1\). In this case \(\gamma \leq 2\alpha+\beta\), and therefore, if \(d_-\neq0\), then
\[
\deg(d_-)=\gamma-\alpha-\beta\leq \alpha.
\]

Although the shifted equality has the form \(v_{\pm1,2m+1}=w_{\pm s,2N}\), the initial value \(\pm s\) need not be the chosen fundamental value \(z_1\). Therefore we cannot apply the bound \((\ref{ineq20})\) to it. Instead, we apply the congruence formula of Lemma~\ref{Lm8} directly to the shifted equality.

The possibility $\gamma=2\alpha+\beta$ is already examined in Case 2.a), so we may assume  \begin{equation}\label{gamma-4.a}\gamma<2\alpha+\beta.\end{equation} If $\alpha=0$, \eqref{gamma-4.a} yields $\gamma<\beta$, a contradiction. Hence, $\alpha>0$. Applying Lemma~\ref{Lm8} to the shifted equality, and multiplying it by $s$, using \(s^2\equiv1\pmod c\), gives
\begin{equation}\label{4a1}
\pm 2am(m+1)+s(2m+1)
\equiv
\pm 2bN^2+\varepsilon(a+b-d_-)N
\pmod c,
\end{equation}
where \(\varepsilon\in\{\pm1\}\). 

We now assume that $\beta<\gamma.$ Since \(N\neq0\), from \eqref{4a1} we obtain
\begin{equation}\label{jed5-new}
d_-=
\pm 2bN+a+b
\mp 2a\frac{m(m+1)}{N}
-\varepsilon s\frac{2m+1}{N}.
\end{equation}

If $\frac{\alpha+\gamma}{2}>\beta,$ then the dominant term in \eqref{jed5-new} is the term containing \(s\). Hence its coefficient must vanish, that is $2m+1=0,$
which is impossible.

Now suppose that $\frac{\alpha+\gamma}{2}<\beta.$
If \(\alpha=\beta\), then this would imply $\frac{\beta+\gamma}{2}<\beta,$
and hence \(\gamma<\beta\), a contradiction. Therefore \(\alpha<\beta\). In this case the dominant term in \eqref{jed5-new} is the term containing \(b\).
Comparison of leading coefficients gives $\pm 2N+1=0,$
which is impossible.

Consequently, the only remaining possibility is $\alpha<\beta=\frac{\alpha+\gamma}{2}.$
Therefore
\begin{equation}\label{4a-deg}
\gamma=2\beta-\alpha.
\end{equation}
Together with \(\gamma\le2\alpha+\beta\), this gives $\beta\le3\alpha.$
We now distinguish two subcases. Assume first that \(2\alpha<\beta\). If \(d_-\neq0\), then, by
\eqref{4a-deg}, $\deg(d_-)=\gamma-\alpha-\beta=\beta-2\alpha.$ Hence, $\deg(ad_-)=\alpha+\deg(d_-)=\beta-\alpha<\beta=\deg(s).$
If \(d_-=0\), the term containing \(d_-\) disappears, so the same conclusion is trivial.

Therefore, after the same simplification as in Case \textbf{2.a)}, with the index \(n\) replaced by \(N\), the divisibility condition by \(s\) gives $s\mid R_{\pm1},$
where $\deg(R_{\pm1})<\deg(s).$ Consequently, $R_{\pm1}=0.$ As in Case \textbf{2.a)}, this gives an equality in which the non-constant
polynomial \(a\) divides a nonzero constant, which is impossible. It remains to consider \(2\alpha\ge\beta\). If \(d_-\neq0\), then, by
\eqref{4a-deg}, $\deg(d_-)=\gamma-\alpha-\beta=\beta-2\alpha\le0.$
Since \(d_-\neq0\), this forces $\deg(d_-)=0,$
and therefore $\beta=2\alpha.$
If \(d_-=0\), the term containing \(d_-\) disappears and the following argument
is even simpler. Denote $M=m(m+1)$. We analyze the divisibility conditions $s \mid R_1$ and \(s\mid R_{-1}\), obtained from \eqref{4a1}, where 
$$ R_1=2a^2(N(N+1)-M)-2aN(N+1)d_--N(2N+1) $$ and $$R_{-1}=2a^2\bigl(N(N-1)-M\bigr)-2aN(N-1)d_--N(2N-1).$$ A direct calculation gives
$(2N-1)R_1-(2N+1)R_{-1}=4a^2(N^2+M)-4aN^2d_-.$ Hence,
$s\mid 4a\bigl(a(N^2+M)-N^2d_-\bigr).$ Since \(s^2=ac+1\), we have \(\gcd(a,s)=1\). Therefore,
\[
s\mid a(N^2+M)-N^2d_-.
\]  Now $\deg(s)=\beta$ and $\deg\bigl(a(N^2+M)\bigr)=\alpha<\beta$. If \(d_-=0\), then clearly $\deg\bigl(a(N^2+M)-N^2d_-\bigr)<\deg(s)$, which is not possible. If \(d_-\neq0\), then $\deg(d_-)=\beta-2\alpha<\beta$,
so again $\deg\bigl(a(N^2+M)-N^2d_-\bigr)<\deg(s).$ Since \(s\) divides this polynomial of smaller degree, it must be zero, so: 
\[
a(N^2+M)=N^2d_-.
\] This is impossible because \(a\neq0\) and \(N^2+M\neq0\). If \(d_-\neq0\), then \(d_-\) is a constant, so the right-hand side is constant, whereas the left-hand side is divisible by the non-constant polynomial \(a\). This is again impossible.

Let $\beta=\gamma$. We have \(\alpha>0\). If \(d_-\neq0\), then $\deg(d_-)=\gamma-\alpha-\beta=-\alpha<0,$ which is impossible. Hence, $d_-=0.$ Thus, in this case, $s\mid R_1 \ \text{and}\ s\mid R_{-1},$
where, with \(M=m(m+1)\),
\[
R_1=
2a^2\bigl(N(N+1)-M\bigr)-N(2N+1)
\]
and
\[
R_{-1}=
2a^2\bigl(N(N-1)-M\bigr)-N(2N-1).
\]
A direct calculation gives
\[
(2N-1)R_1-(2N+1)R_{-1}
=
4a^2(N^2+M).
\]
It follows that $s\mid 4a^2(N^2+M),$ hence
$s\mid N^2+M.$ Obviously, \(N^2+M\) is a nonzero constant, while
$\deg(s)=\frac{\alpha+\gamma}{2}>0.$
This is impossible. Therefore, the case \(\beta=\gamma\) gives no solution.\\

\textbf{Case 4.b)} $v_{2m+1} = w_{2n+1}$, $(z_0, z_1) = (\pm cr\pm st, \pm 1)$ and $\gamma \leq \alpha + 2\beta$.

We obtain an intersection of the same formal type as in Case 3.a). By Lemma \(\ref{preklapanja}\), after a possible change of signs and a shift
of the index in the sequence \((v)_m\), the equality $v_{\pm cr\pm st,2m+1}=w_{\pm1,2n+1}$ can be rewritten in the form $v_{\pm t,2M}=w_{\pm1,2n+1},$
where \(M=m\) or \(M=m+1\). However, the degree conditions overlap only when $\gamma=\alpha+2\beta.$ Hence, it remains to consider $\gamma<\alpha+2\beta.$
 
Let $\beta<\gamma$. Applying the congruence (\ref{3a2}) obtained in Case 3.a), we get
\begin{equation}\label{3a21}
\pm 2aM^2+\varepsilon(a+b-d_-)M\equiv\pm 2bn(n+1)+t(2n+1)\pmod c,
\end{equation}
where \(\varepsilon\in\{\pm1\}\). We have \(\deg (t)=\frac{\beta+\gamma}{2}<\gamma\).
On the other hand, the left-hand side of \((\ref{3a21})\) has degree at most
\(\beta\). Indeed, \(\alpha\leq\beta\) and, if \(d_-\neq0\), then
Lemma~\(\ref{lemma-d_-}\) gives
\(\deg(d_-)=\gamma-\alpha-\beta<\beta\), because
\(\gamma<\alpha+2\beta\). If \(d_-=0\), the same bound is immediate.
Consequently, the congruence
\((\ref{3a21})\) is an equality. Moreover, the coefficient of degree \(\frac{\beta+\gamma}{2}>\beta\) vanishes. Since the leading coefficient of \(t\) is non-zero, this gives
\(2n+1=0\), a contradiction.

Let $\beta = \gamma$. In this remaining case, we return to the original equality
\(v_{2m+1}=w_{2n+1}\), so the shifted index \(M\) will no longer be used.
First we show that the present case necessarily splits into the two special
subcases. Put \(w=cr\pm st\). We have \(cd_-+1=z_0^2\). Since $\beta=\gamma$, Lemma
\(\ref{lemma-d_-}\) implies that either \(d_-=0\), or \(d_-=a\) and
\(\alpha=0\). If \(d_-=0\), then \(w^2=1\), and hence
\((z_0,z_1)=(\pm1,\pm1)\). If \(d_-=a\), then \(w^2=ca+1=s^2\), and hence
\((z_0,z_1)=(\pm s,\pm1)\), with \(\alpha=0\). We apply Lemma~\ref{Lm8} directly to the original equality \(v_{2m+1}=w_{2n+1}\).

First assume that \((z_0,z_1)=(\pm1,\pm1)\). Then \(x_0=y_1=1\), and
\(d_-=0\). By Remark~\(\ref{Rem_dminus}\), we have \(c=a+b\pm2r\),
\(s=\pm(a\pm r)\), \(t=\pm(b\pm r)\), with compatible signs. Hence, $c=\pm s\pm t$. By
Lemma~\(\ref{Lm8}\), after dividing by $c$, we get
\begin{equation}\label{4b1}
\pm1\equiv\pm2btn(n+1)\mp2asm(m+1)+(2n+1)-(2m+1)\pmod c.
\end{equation}

If \(\alpha=\beta=\gamma\), then $\deg s=\deg t=\gamma.$ Moreover, since \((z_0,z_1)=(\pm1,\pm1)\), we may take \(x_0=y_1=1\). Hence, $v_1=c\pm s$, $w_1=c\pm t$, so $\deg v_1\leq\gamma$, $\deg w_1\leq\gamma$. By \eqref{degvm} and \eqref{degwn}, $\deg v_{2m+1}=2m\gamma+\deg v_1$ and $\deg w_{2n+1}=2n\gamma+\deg w_1$. Since \(v_{2m+1}=w_{2n+1}\), we get
\[
        2(m-n)\gamma=\deg w_1-\deg v_1.
\]
The right-hand side lies between \(-\gamma\) and \(\gamma\), while the
left-hand side is a multiple of \(2\gamma\). Therefore, it must be zero.
Since \(\gamma>0\), it follows that \(m=n\).
Substituting \(m=n\) into \((\ref{4b1})\), we get
\begin{equation}\label{4.b-con}
        \pm1 \equiv 2n(n+1)(\mu bt-\lambda as)\pmod c,
\end{equation}
where \(z_0=\lambda\) and \(z_1=\mu\), with
\(\lambda,\mu\in\{\pm1\}\). From Lemma \ref{initial1}, we have $\lambda s\equiv \mu t\pmod c$ so $\lambda as\equiv \mu at\pmod c$. Therefore,
\[
        \mu bt-\lambda as
        \equiv
        \mu(b-a)t
        \pmod c.
\]
Using \(c=a+b+2\eta r\) and \(t=\kappa(b+\eta r)\), with
\(\eta,\kappa\in\{\pm1\}\), put
\(U=a+\eta r\). Then \(b+\eta r=c-U\), so
\(t\equiv-\kappa U\pmod c\), while $b-a\equiv -2U\pmod c.$ Consequently,
$\mu(b-a)t
        \equiv
        2\mu\kappa U^2
        \pmod c.$
Since \(U^2=(a+\eta r)^2=ac+1\), we have \(U^2\equiv1\pmod c\). Thus
\[
        \mu bt-\lambda as\equiv 2\mu\kappa\pmod c,
\]
and hence \(\mu bt-\lambda as\equiv\pm2\pmod c\). Therefore, \eqref{4.b-con} implies
\[
        \pm1\equiv \pm4n(n+1)\pmod c.
\]
Since both sides are constant and \(\gamma>0\), this congruence is an
equality, which has no solution \(n\in\mathbb N_0\). Hence, this case is
impossible.

Let \(\alpha<\beta=\gamma\). Multiplying \((\ref{4b1})\) by \(s\) and using
\(s^2\equiv1\pmod c\) and \(st\equiv\pm1\pmod c\), we obtain a congruence of
the form
\[
\pm s\pm2bn(n+1)\pm2am(m+1)+2s(n-m)\equiv0\pmod c.
\]
Thus, for some \(k\in\mathbb Q(i)\),
\begin{equation}\label{4b2}
    \pm s\pm2bn(n+1)\pm2am(m+1)+2s(n-m)=k(a+b\pm2r).
\end{equation}
Comparing the coefficients of \(b\) and \(r\), respectively, gives
\(k=\pm2n(n+1)\) and \(\pm1+2(n-m)=\pm2k\). In particular, \(k\neq 0\) is an even integer. The second equation is not possible, because of the different parity on the left-hand and right-hand side. 

Now let \((z_0,z_1)=(\pm s,\pm1)\). Then \(\alpha=0\) and
\(d_-=a\). By Remark~\(\ref{remark5}\), it suffices to consider $a=\frac12\left(f-\frac1f\right),\ 
b=f(1-t),\ 
c=-\frac1f(t+1),\ 
r=fs,$
where \(f\in\mathbb Q(i)\setminus\{0\}\). In this case, $u=\frac12\left(f+\frac1f\right).$ Then $t=-fc-1,\ cr+st=-s.$
We set \(z_1=\tau\), where \(\tau\in\{\pm1\}\). By Lemma~\ref{initial1}, $sz_0\equiv tz_1\pmod c.$ Since \(t\equiv-1\pmod c\) and \(s^2\equiv1\pmod c\), it follows that $z_0 \equiv -\tau s\pmod{c}.$ 
Since \(z_0=\pm s\) the choice \(z_0=\tau s\) would imply that $c\mid 2s.$ This is impossible, since $ac+1=s^2.$ Therefore, $z_0=-\tau s.$ By Remark \ref{constant}, $x_0=u$. By Remark \ref{predznaci}, in order to simplify further observations, we take $x_0=\tau u$. 
Moreover, \(t+1=-fc\), and hence \(s^2+t=ac+1+t=(a-f)c=-uc\), while \(t\equiv-1\pmod c\) and \(b=f(1-t)\) imply \(bt\equiv-2f\pmod c\). Since $y_1=1$, by Lemma~\(\ref{Lm8}\), we get
\[
2\tau u(m+1)-2\tau am(m+1)
+4\tau fn(n+1)-(2n+1)\equiv0\pmod c
\] after division by \(c\).
Since the left-hand side is constant and \(\gamma >0\), this is an
equality. Using the expressions for \(a\) and \(u\) and multiplying by \(\tau f\), we obtain
\begin{equation}\label{4b3}
 \bigl((2n+1)^2-m^2\bigr)f^2
 -\tau(2n+1)f+(m+1)^2=0.
\end{equation}
Furthermore, $v_1=\tau\left(\frac{c}{f}-1\right), \  w_1=-\tau+(1-\tau f)c.$ Hence, \(\deg (v_1)=\gamma\). Moreover, \(1-\tau f\neq0\), since otherwise \(f=\tau=\pm1\), which would imply \(a=0\). Hence \(\deg (w_1)=\gamma\), as well. By \eqref{degvm} and \eqref{degwn}, 
 we obtain \(m=2n\). Substituting this into \((\ref{4b3})\), we get $(2m+1)f^2-\tau(m+1)f+(m+1)^2=0.$ The discriminant of the quadratic equation is
$-(m+1)^2(8m+3)$,
so \(8m+3\) is a square in \(\mathbb Q\), and hence in \(\mathbb Z\). This is impossible since $8m+3\equiv3\pmod8.$\\

\textbf{Case 4.c)} $v_{2m+1}=w_{2n+1}$, $(z_0,z_1)=(\pm t, \pm s)$ and $\gamma\geq \alpha+2\beta$.\\
By Lemma \ref{preklapanja}, this case is reduced to the extended form of Case
1.c). The corresponding degree conditions are checked at the end of Case 1.c).
\end{proof}

\end{document}